\pdfoutput=1
\documentclass[reqno]{amsart}

\usepackage{enumerate}
\usepackage{tabto}
\usepackage[mathscr]{euscript}
\usepackage{xcolor}
\usepackage{layout}
\usepackage{fancyhdr}
\usepackage{array}
\usepackage{amsfonts}
\usepackage{amsmath}
\usepackage{amssymb}
\usepackage{mathtools}
\usepackage{graphicx}
\usepackage{bm}
\usepackage{enumitem}
\usepackage{caption} 
\usepackage{color}
\usepackage{csquotes}
\usepackage{bookmark}
\usepackage{float}
\usepackage{multirow}
\usepackage[square,numbers,sort&compress]{natbib}
\usepackage{hyperref}
\hypersetup{colorlinks=true,linkcolor=blue,citecolor=red}
\allowdisplaybreaks

\def\Xint#1{\mathchoice
{\XXint\displaystyle\textstyle{#1}}%
{\XXint\textstyle\scriptstyle{#1}}%
{\XXint\scriptstyle\scriptscriptstyle{#1}}%
{\XXint\scriptscriptstyle\scriptscriptstyle{#1}}%
\!\int}
\def\XXint#1#2#3{{\setbox0=\hbox{$#1{#2#3}{\int}$ }
\vcenter{\hbox{$#2#3$ }}\kern-.6\wd0}}

\def\dashint{\Xint-}

\newtheorem{theorem}{Theorem}[section]
\newtheorem{lemma}[theorem]{Lemma}

\newtheorem{remark}[theorem]{Remark}
\theoremstyle{definition}

\numberwithin{equation}{section}

\newcommand{ \mc }{ \mathcal }
\newcommand{ \mr }{ \mathbb{R} }
\newcommand{ \mf }{ \mathbf{f} }

\newcommand{ \bu }{ \mathbf{u} }
\newcommand{ \bw }{ \mathbf{w} }

\newcommand{ \bv }{ \mathbf{v} }
\newcommand{ \bV }{ \mathbf{V} }

\begin{document}
\title[Regularity for Logarithmic Double Phase Functionals]{Regularity for Vectorial Double Phase Functionals with Two Modulating Coefficients and Logarithmic-Type Growth}

\author{Yumi Kim}\address{Department of Mathematics, Kyungpook National University, Daegu, 41566, Republic of Korea} \email{kym021226@knu.ac.kr} \author{Jehan Oh}\address{Department of Mathematics, Kyungpook National University, Daegu, 41566, Republic of Korea} \email{jehan.oh@knu.ac.kr}

\date{\today.}
\thanks{This work is supported by the National Research Foundation of Korea (NRF) grant funded by the Korea government [Grant Nos. RS-2025-00555316 and RS-2025-25415411].}

\begin{abstract}
We study vector-valued local minimizers of double phase functionals with two modulating coefficients and logarithmic-type growth. The energy density is $a(x)|D\bu|^p+b(x)|D\bu|^p\log(e+|D\bu|)$, where the nonnegative coefficients $a(\cdot)$ and $b(\cdot)$ may vanish but their sum is bounded away from zero. Assuming that $a(\cdot)$ is uniformly continuous and that $b(\cdot)$ satisfies a vanishing log-H\"older continuity condition, we prove that minimizers are locally H\"older continuous for every exponent in $(0,1)$. If both coefficients are H\"older continuous, then the gradient of a minimizer is locally H\"older continuous. The proof combines Gehring-type higher integrability with a comparison argument.
\end{abstract}

\maketitle

\section{\bf Introduction}\label{section 1}
Functionals with non-standard growth have been studied extensively in the calculus of variations, starting from the works of Marcellini \cite{Marcellini1989,Marcellini1991} on functionals with $(p,q)$-growth. A prominent example is the double phase functional
$$
w\mapsto\int_{\Omega}\left(|Dw|^p+a(x)|Dw|^q\right)dx,
$$
where $1<p<q$ and $a:\Omega\to[0,\infty)$ is a modulating coefficient. This functional was introduced by Zhikov \cite{Zhikov1986,Zhikov1993,Zhikov1995,Zhikov1997} in the context of homogenization and the Lavrentiev phenomenon. The coefficient $a(\cdot)$ determines where the $q$-growth term is active, and the regularity of minimizers depends on the regularity of $a(\cdot)$ and on the size of the gap between $p$ and $q$. The regularity theory for double phase functionals was developed by Colombo and Mingione \cite{Colombo2015,Colombo215a} and by Baroni, Colombo, and Mingione \cite{BCM2015,Baroni2018}.

Functionals with logarithmic growth lie at the borderline between $p$-growth and $(p,q)$-growth. A basic example is
$$
w\mapsto\int_{\Omega}|Dw|^p\log(e+|Dw|)\,dx,
$$
whose integrand grows faster than $|Dw|^p$ but slower than $|Dw|^q$ for every $q>p$. Such functionals have been studied in the regularity theory of variational problems, see for instance \cite{Giannetti2013,Ok2018}. Logarithmic growth is also formally related to the limit $q\to p$ of the double phase functional, since
$$
\frac{|Dw|^q-|Dw|^p}{q-p}\longrightarrow |Dw|^p\log|Dw| \qquad \text{as } q\to p.
$$
Replacing the $q$-growth term by a logarithmic perturbation of the $p$-growth term leads to the functional
$$
w\mapsto\int_{\Omega}\left(|Dw|^p+a(x)|Dw|^p\log(e+|Dw|)\right)dx,
$$
which was studied by Baroni, Colombo, and Mingione \cite{Baroni2015} as a borderline case of double phase functionals, see also \cite{Byun2017}. In this setting, conditions of logarithmic type on the modulus of continuity of $a(\cdot)$ appear naturally. More general double phase functionals, whose two terms have Orlicz growth, were studied by Byun and Oh \cite{BO2020}.

Another direction is to allow more than one modulating coefficient. This is natural from a modeling point of view, since different parts of the energy may be affected by different coefficients, and such functionals appear, for instance, in multi-phase problems \cite{DO2019}. In particular, separate coefficients for the two terms of a double phase functional allow each term to vanish independently of the other. Kim and Oh \cite{KO2024} studied the functional
$$
w\mapsto\int_{\Omega}\left(a(x)|Dw|^p+b(x)|Dw|^q\right)dx, \qquad 1<p<q,
$$
with two nonnegative modulating coefficients whose sum is bounded away from zero, where $a(\cdot)$ is uniformly continuous, $b(\cdot)$ is H\"older continuous, and $q$ is suitably close to $p$. A related problem with nearly linear growth, in which a nearly linear logarithmic term is combined with a higher power term, was considered in \cite{KO2025}. In the present paper, we keep this two-coefficient structure, but both terms have the same power growth with exponent $p>1$ and the second term carries an additional logarithmic factor.

For vector-valued minimizers, Ragusa and Tachikawa \cite{RT2020} proved the H\"older continuity of the gradients of minimizers of double phase functionals with variable exponents, and De Filippis and Mingione \cite{DM2020} obtained partial regularity for sphere-valued minimizers of non-uniformly elliptic integrals. In \cite{RT2024}, Ragusa and Tachikawa studied vector-valued minimizers of a variable exponent version of the borderline functional above, in which $p$ is replaced by a variable exponent $p(\cdot)$ and the power term keeps the fixed coefficient $1$. They proved the H\"older continuity of minimizers and of their gradients under suitable continuity assumptions on $a(\cdot)$ and $p(\cdot)$. Compared with \cite{RT2024}, the exponent $p>1$ is fixed in the present paper, but the power term also carries a modulating coefficient, and either of the two coefficients may vanish as long as their sum stays bounded away from zero. Moreover, for the H\"older continuity of minimizers, the coefficient of the power term is only assumed to be uniformly continuous, while the coefficient of the logarithmic term is required to be vanishing log-H\"older continuous, that is, its modulus of continuity is $o(1/\log(1/r))$. Our focus is the regularity theory for this two-coefficient structure.

More precisely, we consider the functional
\begin{equation}\label{main functional}
\mathcal P_{\log}(\bu,\Omega)\coloneq\int_{\Omega}\left[a(x)|D\bu|^p+b(x)|D\bu|^p\log(e+|D\bu|)\right]dx,
\end{equation}
defined on $W^{1,1}(\Omega;\mr^N)$, where $N\geq1$, $p>1$, and $\Omega\subset\mr^n$ is a bounded open set. Here $a,b : \overline{\Omega} \rightarrow \mathbb{R}$ are nonnegative bounded functions. We assume that $a(\cdot)$ is uniformly continuous on $\overline{\Omega}$, that is, there exists a modulus of continuity $\omega_{a}$ such that
$$
 |a(x)-a(y)| \leq \omega_{a}\left( |x-y| \right) \; \text{ for every } x,y \in \overline{\Omega}
$$
and
\begin{equation}\label{cond : omega_a}
    \lim_{r \rightarrow 0} \omega_a(r) = 0.
\end{equation}
We also assume that there exists $\mu >0$ such that 
\begin{equation}\label{cond : a,b}
    a(x) + b(x) \geq \mu \quad \text{ for all } x \in \Omega.
\end{equation}

The two modulating coefficients play different roles. The coefficient $a(\cdot)$ multiplies the $p$-growth term, while $b(\cdot)$ multiplies the logarithmic term. Either coefficient may vanish, but \eqref{cond : a,b} prevents both of them from vanishing at the same point. The logarithmic term requires an additional condition on the modulus of continuity $\omega_b$ of $b(\cdot)$, since the quantity
$$
\omega_b(R)\log\left(\frac{1}{R}\right)
$$
appears in the comparison estimates on a ball of radius $R$. This condition is expressed through the quantity $\ell$ defined in \eqref{cond : ell} below.

We say that $\bu \in W^{1,1}(\Omega;\mr^N)$ is a local minimizer of $\mathcal{P}_{\log}$ if
$$
a(x)|D\bu|^p+b(x)|D\bu|^p\log(e+|D\bu|) \in L^1(\Omega)
$$
and $\mathcal{P}_{\log}(\bu,\operatorname{supp}(\bu-\bw)) \le \mathcal{P}_{\log}(\bw,\operatorname{supp}(\bu-\bw))$ for every $\bw \in W^{1,1}(\Omega;\mr^N)$ with $\operatorname{supp}(\bu-\bw) \Subset \Omega$. By \eqref{cond : a,b}, the integrand in \eqref{main functional} is bounded from below by $\mu|D\bu|^p$, so every local minimizer satisfies $D\bu \in L^p(\Omega;\mr^{N\times n})$. Since our results are local, we may therefore assume $\bu \in W^{1,p}(\Omega;\mr^N)$ in Sections \ref{section 4} and \ref{section 5}.

The main result of this paper is the following.

\begin{theorem}\label{thm : main theorem}
    Let $\bu \in W^{1,1}(\Omega;\mr^N)$ be a local minimizer of the functional $\mathcal{P}_{\log}$ defined in \eqref{main functional}, where $a(\cdot)$ and $b(\cdot)$ satisfy \eqref{cond : a,b} and $a(\cdot)$ is uniformly continuous on $\overline{\Omega}$.
    Let $\omega_b(\cdot)$ be a modulus of continuity of $b(\cdot)$ in the sense that
    \begin{equation}\label{cond : modulus of b}
        |b(x)-b(y)| \leq \omega_{b}\left( |x-y| \right) \; \text{ for every } x,y \in \Omega
    \end{equation}
    and denote 
    \begin{equation}\label{cond : ell}
        \limsup_{r \rightarrow 0} \omega_{b}(r)\log\left( \dfrac{1}{r} \right) \eqcolon \ell.
    \end{equation}
    Then the following statements hold.
    \begin{enumerate}[label=\textup{(\roman*)}]
        \item If $\ell =0$, then $\bu \in C_{\operatorname{loc}}^{0, \beta}(\Omega ; \mr^N)$ for every $\beta \in (0,1)$.
        \item If $a(\cdot)$ and $b(\cdot)$ are H\"older continuous, then $D\bu$ is locally H\"older continuous in $\Omega$.
    \end{enumerate}
\end{theorem}

The proof is based on a comparison argument with minimizers of frozen functionals. On a ball $B_R\Subset\Omega$, we replace $a(\cdot)$ and $b(\cdot)$ by their infima $a_i(R)$ and $b_i(R)$ over $B_R$ and consider the minimizer $\bv$ of the corresponding autonomous functional on $B_{R/2}$ with boundary values $\bu$. The regularity estimates for such frozen functionals in Theorem \ref{Thm 3.1} are derived from the theory of functionals with $\varphi$-growth, see \cite{Lieberman1991} for the scalar case and \cite{Diening2009} for the vectorial case, and their constants and H\"older exponent do not depend on the frozen coefficients. In the vectorial case, this theory is available under an additional H\"older-type condition on the second derivative of the Orlicz function, see \cite[Assumption 2.2]{Diening2009}. For general non-autonomous functionals with Orlicz growth, H\"ast\"o and Ok \cite[Remark 1.7]{HO2022} pointed out that it is unclear whether this condition follows from their assumptions, and they left the vectorial case as a topic for future research. We do not address this general question here, but we verify the condition directly for the frozen functionals arising from \eqref{main functional} in the proof of Theorem \ref{Thm 3.1}.

Although $a_i(R)$ or $b_i(R)$ may vanish, their sum stays above $\mu/2$ on sufficiently small balls by \eqref{ineq : ai+bi}, so that no separate lower bound on either coefficient is needed. Using the higher integrability of Section \ref{section 4}, which relies on Gehring's lemma \cite{Gehring1973} as in \cite{Giaquinta1982}, we then obtain in Lemma \ref{lem : comparison} a comparison estimate between $\bu$ and $\bv$ whose error is controlled by
$$
\omega_a(R)+\omega_b(R)\log\left(\frac{1}{R}\right).
$$
The oscillation of $a(\cdot)$ enters this quantity linearly, while the oscillation of $b(\cdot)$ is weighted by $\log(1/R)$ because of the logarithmic factor. The smallness of this quantity on small balls yields decay estimates for the energy of $\bu$, and the local H\"older continuity of $\bu$ and of $D\bu$ then follows from Morrey- and Campanato-type estimates.

The paper is organized as follows. In Section \ref{section 2}, we introduce the notation and collect some preliminary results. In Section \ref{sectioon 3}, we study the frozen autonomous problem and establish the estimates used in the comparison argument. Section \ref{section 4} is devoted to the higher integrability of local minimizers. In Section \ref{section 5}, we prove Theorem \ref{thm : main theorem} by combining the comparison estimates with a decay argument.

\section{\bf Notation and preliminaries}\label{section 2}
We denote by $c$ a generic positive constant, which may change from line to line. Specific constants are denoted by $c_1$, $c_*$, $\bar{c}$, and so on. We denote by
$$B_r(x_0) \coloneq \left\{ x \in \mathbb{R}^n : |x-x_0| < r \right\}$$
the open ball with center $x_0$ and radius $r>0$. When the center is clear from the context, we simply write $B_r \equiv B_r(x_0)$. If $\mathcal{M} \subset \mathbb{R}^n$ is a measurable set with $0<|\mathcal{M}|<\infty$ and $f: \mathcal{M} \rightarrow \mathbb{R}^k$, $k \geq 1$, is an integrable function, we write
$$(f)_{\mathcal{M}} \equiv \dashint_{\mathcal{M}} f(x) \, dx \coloneq \dfrac{1}{|\mathcal{M}|} \int_{\mathcal{M}} f(x)\, dx.$$
For $t>0$, $\log t$ denotes the natural logarithm of $t$, and for $\gamma >0$ we write $\log^{\gamma}(e+t) \coloneq [ \log (e+t) ]^{\gamma}$.

Consider a convex function $\varphi : [0,\infty) \rightarrow [0,\infty)$ such that
\begin{equation}\label{define : convex function}
    \begin{cases}
        &\varphi \in C^1([0,\infty)) \cap C^2((0,\infty)), \, \displaystyle\lim_{t \rightarrow \infty} \varphi'(t) = \infty,\\ 
        &\varphi(0) = \varphi'(0)=0,\\ 
        &\varphi(t)=0 \Longleftrightarrow t=0.
    \end{cases}
\end{equation}
We define the vector field $\bV_{\varphi} : \mathbb{R}^{N \times n} \rightarrow \mathbb{R}^{N \times n}$ by 
\begin{equation}\label{def : vector field}
    \bV_{\varphi}(\mathbf{P}) \coloneq \left( \dfrac{\varphi'(|\mathbf{P}|)}{|\mathbf{P}|} \right)^{\frac{1}{2}} \mathbf{P}
\end{equation}
for $\mathbf{P} \in \mr^{N \times n}$.
By \eqref{define : convex function}, $\bV_{\varphi}$ is a bijection of $\mr^{N \times n}$.
Under the assumption
\begin{equation}\label{cond : mono map}
    \dfrac{1}{c_{\varphi}} \leq \dfrac{\varphi''(t)t}{\varphi'(t)} \leq c_{\varphi}
    \quad \text{for all } t>0 \text{ and some } c_{\varphi} \geq 1,
\end{equation}
$\bV_{\varphi}$ characterizes the monotonicity of the mapping
$\mathbf{P} \mapsto \left( \frac{\varphi'(|\mathbf{P}|)}{|\mathbf{P}|} \right) \mathbf{P}$. Indeed, for $\mathbf{Q}, \mathbf{R} \in \mr^{N \times n}$ we have
\begin{equation}\label{eq : monotonicity V}
    \dfrac{1}{c} \left|\bV_{\varphi}(\mathbf{Q})-\bV_{\varphi}(\mathbf{R})\right|^2 \leq \left\langle \dfrac{\varphi'(|\mathbf{Q}|)}{|\mathbf{Q}|}\mathbf{Q} - \dfrac{\varphi'(|\mathbf{R}|)}{|\mathbf{R}|}\mathbf{R} , \mathbf{Q}- \mathbf{R}\right\rangle \leq c\left|\bV_{\varphi}(\mathbf{Q})-\bV_{\varphi}(\mathbf{R}) \right|^2,
\end{equation}
with a constant $c \geq 1$ depending on $n$, $N$, and $c_{\varphi}$.
Moreover, by \cite[Lemma 2.4]{Diening2009}, there exists a constant $c$ depending only on $n$, $N$, and the constant $c_{\varphi}$ in \eqref{cond : mono map} such that
\begin{equation}\label{ineq : 2.5}
    \varphi''(|\mathbf{Q}| + |\mathbf{R}|) |\mathbf{Q}-\mathbf{R}|^2 \leq c\,|\bV_{\varphi}(\mathbf{Q}) - \bV_{\varphi}(\mathbf{R})|^2.
\end{equation}
The quantity $\left| \bV_{\varphi}(\mathbf{P}) \right|^2$ is comparable to $\varphi(|\mathbf{P}|)$. More precisely, there exists a constant $c \equiv c(c_{\varphi})$ such that, for every $\mathbf{P} \in \mr^{N \times n}$,
\begin{equation}\label{ineq  : comparable}
    \dfrac{1}{c}\left| \bV_{\varphi}(\mathbf{P}) \right|^2 \leq \varphi(|\mathbf{P}|) \leq c \left| \bV_{\varphi}(\mathbf{P}) \right|^2.
\end{equation}
This follows from \eqref{cond : mono map} together with the estimate $\frac{\varphi(t)}{c(c_{\varphi})} \leq \varphi'(t)t \leq c(c_{\varphi})\varphi(t)$ for all $t \geq 0$, which can be obtained via integration by parts. 
For the particular choice $\varphi(t)=\frac{t^p}{p}$, we write $\bV_p(\cdot) \equiv \bV_{\varphi}$, that is,
\begin{equation}
    \bV_p(\mathbf{P}) \coloneq |\mathbf{P}|^{\frac{p-2}{2}}\mathbf{P}.
\end{equation}
Similarly, for $\varphi(t)=\frac{1}{p}t^p\log(e+t)$, we write $\bV_{\log}(\cdot) \equiv \bV_{\varphi}$, that is, 
\begin{equation}
    \bV_{\log}(\mathbf{P}) \coloneq \left( 
    |\mathbf{P}|^{p-2} \log (e+|\mathbf{P}|)+\dfrac{|\mathbf{P}|^{p-1}}{p(e+|\mathbf{P}|)}
    \right)^{\frac{1}{2}}\mathbf{P}.
\end{equation}
In this setting, the following elementary inequality holds for all $\mathbf{Q}, \mathbf{R} \in \mr^{N \times n}$:
\begin{equation}\label{ineq : |V_p|}
    \dfrac{1}{c} (|\mathbf{Q}| + |\mathbf{R}|)^{\frac{p-2}{2}} |\mathbf{Q} - \mathbf{R}|
    \leq \left| \bV_p(\mathbf{Q}) - \bV_p(\mathbf{R}) \right|
    \leq c(|\mathbf{Q}| + |\mathbf{R}|)^{\frac{p-2}{2}} |\mathbf{Q} - \mathbf{R}|,
\end{equation}
where $c$ depends only on $n$, $N$, and $p$, see \cite{Giusti2003}.
In particular, if $p \geq 2$, then we have
\begin{equation}\label{2.10}
    |\mathbf{Q} - \mathbf{R}|^p \leq c\left| \bV_p(\mathbf{Q}) - \bV_p(\mathbf{R}) \right|^2
\end{equation}
for a constant $c$ depending on $n$, $N$, and $p$.

\begin{remark}\label{cond : conv ft prop}
    If $\varphi_1$ and $\varphi_2$ are as above and satisfy \eqref{cond : mono map} with constants $c_{\varphi_1}$ and $c_{\varphi_2}$, then their sum 
    $\varphi \coloneq \varphi_1 + \varphi_2$ satisfies \eqref{cond : mono map} with $c_{\varphi} \coloneq 2\max\{c_{\varphi_1}, c_{\varphi_2}\}$. Indeed,
    \begin{equation*}
    \dfrac{\varphi''_1(t)t}{\varphi'_1(t) + \varphi'_2(t)}
    + \dfrac{\varphi''_2(t)t}{\varphi'_1(t) + \varphi'_2(t)}
    \leq \dfrac{\varphi''_1(t)t}{\varphi'_1(t)}
    +\dfrac{\varphi''_2(t)t}{\varphi'_2(t)}
    \leq c_{\varphi_1} + c_{\varphi_2}
    \end{equation*} and
    \begin{equation*}
    \dfrac{\varphi''_1(t)t}{\varphi'_1(t) + \varphi'_2(t)}
    + \dfrac{\varphi''_2(t)t}{\varphi'_1(t) + \varphi'_2(t)}
    \geq \dfrac{\left[ \varphi''_1(t) + \varphi''_2(t) \right]t}{2\max \{\varphi'_1(t), \varphi'_2(t) \}}
    \geq \dfrac{1}{2} \min \left\{ \dfrac{1}{c_{\varphi_1}}, \dfrac{1}{c_{\varphi_2}} \right\}.
    \end{equation*}
\end{remark} 

We will apply the previous results to the function
\begin{equation}\label{rem : 2.1}
    \varphi(t) \coloneq a_0 t^p + b_0 t^p \log(e+t) ,
\end{equation}
where $p>1$ and $a_0 ,\,b_0 \geq 0$ satisfy $a_0 + b_0 > 0$.
This function satisfies \eqref{define : convex function}, and by Remark \ref{cond : conv ft prop} it satisfies \eqref{cond : mono map} with a constant $c_{\varphi} \equiv c_{\varphi}(p)$ independent of $a_0$ and $b_0$.

We will use the following Sobolev--Poincar\'e inequality, which can be found in \cite[Theorem 7]{Baroni2015a}.
\begin{lemma}\label{rem: 2.2}
    Let $\varphi : [0,\infty) \rightarrow [0,\infty)$ satisfy \eqref{define : convex function} and \eqref{cond : mono map}, and let $B_R \equiv B_R(x_0) \subset \Omega$. Then there exist an exponent $d_1 \in (0,1)$ and a constant $c$, both depending only on $n$, $N$, and $c_{\varphi}$, such that 
    \begin{equation}\label{ienq : Sobolev}
        \dashint_{B_R} \varphi\left( \dfrac{|\mf-(\mf)_{B_R}|}{R} \right)\, dx \leq c\left( \dashint_{B_R} \left[ \varphi(|D\mf|) \right]^{d_1}\, dx \right)^{\frac{1}{d_1}}
    \end{equation}
    for every function $\mf \in  W^{1,1}(\Omega;\mr^N)$ with $\int_{B_R} \varphi (|D\mf|)\, dx < \infty$.
\end{lemma}
For a bounded open set $\Omega \subset \mr^n$ and a strictly increasing convex function $\varphi$ with $\displaystyle\lim_{t \rightarrow 0} \frac{\varphi(t)}{t} =0 $ and $\displaystyle\lim_{t \rightarrow \infty} \frac{\varphi(t)}{t} = \infty$, we consider the Orlicz space $L^{\varphi}(\Omega;\mr^k)$ with the Luxemburg norm
\begin{equation}
    \|\mf\|_{\varphi} = \inf \left\{ \lambda >0 : \dashint_{\Omega} \varphi \left( \dfrac{|\mf(x)|}{\lambda} \right) \, dx \leq 1 \right\}
\end{equation}
for measurable functions $\mf : \Omega \rightarrow \mr^k$. In particular, if $\varphi(t) = t^p$, the above quantity corresponds to the averaged $L^p$ norm. For $\varphi(t) = t^p \log^{\gamma}(e+t)$ with $p \geq 1$ and $\gamma > 0$, the Orlicz space $L^{\varphi}(\Omega;\mr^k)$ is denoted by $L^p\log^{\gamma}L(\Omega;\mr^k)$. It consists of all measurable functions $\mf$ satisfying
$$
\int_{\Omega} |\mf|^p \log^{\gamma} (e+|\mf|) \, dx < \infty.
$$
It is well known that the $L \log^{\gamma}L$ norm, which we denote by $\left| \mf \right|_{L \log^{\gamma}L}$, is controlled by the averaged $L^q$ norm for every $q >1$, namely
$$
\left| \mf \right|_{L \log^{\gamma}L(B_R)} \leq c(\gamma,q)\left( \dashint_{B_R}|\mf|^q\, dx \right)^{\frac{1}{q}}.
$$
Moreover, by \cite[Section 8]{Iwaniec1999}, $\left|\mf\right|_{L \log^{\gamma}L(B_R)}$ is equivalent to the quantity
$$
\dashint_{B_R} |\mf|\log^{\gamma}\left( e + \dfrac{|\mf|}{(|\mf|)_{B_R}} \right)\, dx
$$ 
up to constants depending only on $n$, $k$, and $\gamma$. Consequently, for all $q>1$ and $\mf \in L^q(B_R;\mr^k)$, we have
\begin{equation}\label{L log L}
    \dashint_{B_R}|\mf| \log^{\gamma}\left( e+\dfrac{|\mf|}{(|\mf|)_{B_R}} \right) \, dx \leq c(n, k, \gamma, q) \left( \dashint_{B_R} |\mf|^q\, dx \right)^{\frac{1}{q}}.
\end{equation}

We next recall two iteration lemmas. The first one can be found in \cite[Lemma 7.3]{Giusti2003}.
\begin{lemma}\label{lemma 2.5}
    Let $\phi : [0, \tilde{R}] \rightarrow [0,\infty)$ be a nondecreasing function satisfying, for some $\epsilon \geq 0$,
    $$
    \phi(\rho) \leq \tilde{c} \left[ \left( \dfrac{\rho}{R} \right)^n + \epsilon \right] \phi(R) \quad \text{whenever } 0 < \rho \leq R \leq \tilde{R}.
    $$
    Then for every $\delta \in (0,n)$ there exists $\overline{\epsilon} \equiv \overline{\epsilon}(n,\delta,\tilde{c}) >0$ such that if $\epsilon \leq \overline{\epsilon}$, then
    $$
    \phi(\rho) \leq \overline{c}\left( \dfrac{\rho}{R}  \right)^{n-\delta} \phi(R)
    $$
    whenever $0 < \rho \leq R \leq \tilde{R}$, for a constant $\overline{c} \equiv \overline{c}(n,\delta,\tilde{c}) >0$.
\end{lemma}
The second one is \cite[Lemma 6.1]{Giusti2003}.
\begin{lemma}\label{lem : iteration lemma}
    Let $h : [r_1,r_2] \rightarrow [0,\infty)$ be a bounded function such that
    $$
    h(r) \leq \theta h(s) + \dfrac{A}{(s-r)^{\kappa}} \quad \text{for every } r_1 \leq r < s \leq r_2,
    $$
    where $\theta \in (0,1)$, $A \geq 0$, and $\kappa >0$. Then
    $$
    h(r_1) \leq \dfrac{c(\theta,\kappa)A}{(r_2-r_1)^{\kappa}}.
    $$
\end{lemma}
The self-improving property of reverse H\"older inequalities is described by the following lemma, which goes back to Gehring \cite[Lemma 3]{Gehring1973}, see also \cite[Section 6.4]{Giusti2003}.
\begin{lemma}\label{Lem : Gehring lemma}
    Let $f \in L^1(\Omega)$ be such that
    $$
    \dashint_{B_{R/2}} |f| \, dx \leq \tilde{c}\left( \dashint_{B_R} |f|^d \, dx \right)^{\frac{1}{d}}
    $$
    for some exponent $d \in (0,1)$, some constant $\tilde{c} \geq 1$, and every ball $B_R \subset \Omega$ with radius $R \leq R_0$. Then there exists an exponent $\delta_g \equiv \delta_g(d, \tilde{c})>0$ such that $f \in L_{\operatorname{loc}}^{1+\delta_g}(\Omega)$ and 
    $$
    \dashint_{B_{R/2}} |f|^{1+\delta_g}\, dx \leq c(d,\tilde{c},R_0) \left( \dashint_{B_R} |f| \, dx \right)^{1+\delta_g}
    $$
    for every ball $B_R \subset \Omega$ with radius $R \leq R_0$.
\end{lemma}
Finally, we recall that, for all $x, y \geq 0$ and $A \geq 1$,
\begin{equation}\label{cond : log}
    \begin{aligned}
        &\log(e+xy) \leq \log(e+x) + \log(e+y)\\
        &\log(e+Ax) \leq A \log(e+x)
    \end{aligned}
\end{equation}

\section{\bf Estimates for frozen functionals}\label{sectioon 3}
We consider the functional
\begin{equation}\label{def : functional P_0}
    \mathcal{P}_0(\bw, \Omega) \coloneq \int_{\Omega} \left[ a_0 |D\bw|^p + b_0 |D\bw|^p \log(e+|D\bw|) \right] \, dx,
\end{equation}
where $a_0, b_0 \geq 0$ are constants with $a_0+b_0>0$.

\begin{theorem}\label{Thm 3.1}
    Let $\bv \in W^{1,p}(\Omega;\mr^N)$ be a local minimizer of the functional $\mathcal{P}_0$ defined in \eqref{def : functional P_0}. There exists $\tilde{\alpha} \in (0,1)$, depending only on $n$, $N$, and $p$ and in particular independent of $a_0$, $b_0$, and $\bv$, such that $D\bv \in C_{\operatorname{loc}}^{0,\tilde{\alpha}}(\Omega; \mr^{N \times n})$. Moreover, whenever $B_R \subset \Omega$,
    we have the following inequalities:
    \begin{equation}\label{ineq : sup, int}
    \begin{aligned}
        &\sup_{B_{R/2}} \left( a_0 |D\bv|^p + b_0 |D\bv|^p \log(e + |D\bv|) \right) \\
        &\hspace{10em} \leq c\, \dashint_{B_R} \left( a_0 |D\bv|^p + b_0 |D\bv|^p \log(e+|D\bv|)  \right) \, dx
    \end{aligned}
    \end{equation}
    and, for every $0 < \rho \leq R$,
    \begin{equation}\label{ineq : domain expansion}
    \begin{aligned}
        &(a_0+b_0) \dashint_{B_{\rho}} |D\bv-(D\bv)_{B_{\rho}}|^p \, dx \\
        &\hspace{6em} \leq c \left( \dfrac{\rho}{R} \right)^{\tilde{\alpha}p} \dashint_{B_R} \left(  a_0|D\bv|^p + b_0 |D\bv|^p \log(e+|D\bv|)  \right)\, dx,
    \end{aligned}
    \end{equation}
    where the positive constant $c$ depends only on $n$, $N$, and $p$.
\end{theorem}
\begin{proof}
    Let $\varphi(t) \coloneq a_0 t^p + b_0 t^p \log(e+t)$, and let $\bV_{\varphi}$ be the vector field defined in \eqref{def : vector field}. By Remark \ref{cond : conv ft prop}, $\varphi$ satisfies \eqref{cond : mono map} with $c_{\varphi} \equiv c_{\varphi}(p)$, independently of $a_0$ and $b_0$.
    A direct computation also shows that $t|\varphi'''(t)| \leq c(p)\,\varphi''(t)$ for all $t>0$. Integrating this inequality, we see that $\varphi''(\tau)$ and $\varphi''(t)$ are comparable whenever $\frac{t}{2} \leq \tau \leq \frac{3t}{2}$, and hence
    $$
    |\varphi''(t+s)-\varphi''(t)| \leq c(p)\,\varphi''(t)\,\frac{|s|}{t} \quad \text{whenever } t>0 \text{ and } |s| \leq \frac{t}{2}.
    $$
    Thus $\varphi$ also satisfies the H\"older-type condition on $\varphi''$ in \cite[Assumption 2.2]{Diening2009}, again with constants depending only on $p$.
    Consequently, \eqref{ineq : sup, int} follows from \cite[Lemma 5.8]{Diening2009}, and it remains to prove the decay estimate \eqref{ineq : domain expansion}.
    Since the constants in \cite{Diening2009} depend only on $n$, $N$, and the constants in these two conditions, \cite[Theorem 6.4]{Diening2009} gives the excess decay estimate
    \begin{equation}
        \dashint_{B_{\rho}} |\bV_{\varphi}(D\bv) -  \left(\bV_{\varphi}(D\bv)\right)_{B_{\rho}} |^2 \, dx
        \leq c\left( \dfrac{\rho}{R} \right)^{\alpha_1} \dashint_{B_R} |\bV_{\varphi}(D\bv) -  \left(\bV_{\varphi}(D\bv)\right)_{B_R} |^2 \, dx,
    \end{equation}
    for every $\rho \in (0,R]$, where $c \geq 1$ and $\alpha_1 \in (0,1)$ depend only on $n$, $N$, and $p$.
    Moreover, for $t>0$ we have
    $$
    \varphi''(t) = a_0 p(p-1)t^{p-2} + b_0\left[ p(p-1)t^{p-2}\log(e+t) + \frac{t^{p-1}}{e+t}\left(2p-\frac{t}{e+t}\right) \right].
    $$
    Since $\log(e+t) \geq 1$ and $2p-\frac{t}{e+t} > 0$, setting $\mu_0 \coloneq a_0+b_0 > 0$, we obtain
    \begin{equation}\label{ineq : phi lower}
        \varphi''(t) \geq p(p-1)\mu_0 t^{p-2} \quad \text{and} \quad \varphi(t) \geq \mu_0 t^p \quad \text{for all } t>0.
    \end{equation}
    It suffices to consider the case $\rho \leq \frac{R}{2}$. Indeed, if $\frac{R}{2} < \rho \leq R$, then \eqref{ineq : phi lower} gives
    $$
    \mu_0 \dashint_{B_{\rho}} |D\bv-(D\bv)_{B_{\rho}}|^p \, dx \leq 2^{p+n} \mu_0 \dashint_{B_R} |D\bv|^p \, dx \leq 2^{p+n} \dashint_{B_R} \varphi(|D\bv|) \, dx,
    $$
    which implies \eqref{ineq : domain expansion} because $\frac{\rho}{R} > \frac{1}{2}$.
    Since $\bV_{\varphi}$ is a bijection of $\mr^{N \times n}$, there exists a unique matrix $\mathbf{A} \in \mr^{N \times n}$ such that
    \begin{equation}\label{def : bijection}
    \bV_{\varphi}(\mathbf{A}) = \left(\bV_{\varphi}(D\bv)\right)_{B_{\rho}}.
    \end{equation}
    We distinguish two cases. If $p \geq 2$, then \eqref{ineq : 2.5}, \eqref{def : bijection}, and \eqref{ineq : phi lower} give
    \begin{align*}
        &\dashint_{B_{\rho}} |D\bv - (D\bv)_{B_{\rho}}|^p \, dx \leq 2^p \dashint_{B_{\rho}} |D\bv-\mathbf{A}|^p \, dx\\
        &\hspace{1cm}\leq c\,\dashint_{B_{\rho}} (|D\bv| + |\mathbf{A}|)^{p-2} |D\bv - \mathbf{A}|^2 \, dx \leq \frac{c}{\mu_0}\, \dashint_{B_{\rho}} \varphi''(|D\bv|+|\mathbf{A}|) |D\bv-\mathbf{A}|^2\, dx\\
        &\hspace{1cm}\leq \frac{c}{\mu_0}\, \dashint_{B_{\rho}} |\bV_{\varphi}(D\bv) -\bV_{\varphi}(\mathbf{A})|^2 \, dx = \frac{c}{\mu_0}\, \dashint_{B_{\rho}} |\bV_{\varphi}(D\bv) -  \left(\bV_{\varphi}(D\bv)\right)_{B_{\rho}} |^2 \, dx\\
        &\hspace{1cm}\leq \frac{c}{\mu_0}\left( \dfrac{\rho}{R} \right)^{\alpha_1} \dashint_{B_R} |\bV_{\varphi}(D\bv) - \left(\bV_{\varphi}(D\bv)\right)_{B_R} |^2 \, dx \\
        &\hspace{1cm}\leq \frac{c}{\mu_0}\left( \dfrac{\rho}{R} \right)^{\alpha_1} \dashint_{B_R} (|\bV_{\varphi}(D\bv)|^2 + |\left(\bV_{\varphi}(D\bv)\right)_{B_R} |^2) \, dx\\
        &\hspace{1cm}\leq \frac{c}{\mu_0}\left( \dfrac{\rho}{R} \right)^{\alpha_1} \dashint_{B_R} (a_0|D\bv|^p + b_0 |D\bv|^p \log(e+|D\bv|))\, dx .
    \end{align*}
    Therefore, \eqref{ineq : domain expansion} holds with $\tilde{\alpha} \coloneq \frac{\alpha_1}{p}$. If $1 < p < 2$, then H\"older's inequality, \eqref{ineq  : comparable}, \eqref{ineq : phi lower}, and \eqref{ineq : sup, int} give
    \begin{align*}
        &\dashint_{B_{\rho}} |D\bv-(D\bv)_{B_{\rho}}|^p \, dx\\
        &\hspace{1mm}\leq 2^p \dashint_{B_{\rho}} (|D\bv|+|\mathbf{A}|)^{\frac{p(p-2)}{2}} |D\bv-\mathbf{A}|^p (|D\bv| + |\mathbf{A}|)^{\frac{p(2-p)}{2}} \, dx\\
        &\hspace{1mm} \leq c \left( \dashint_{B_{\rho}} (|D\bv| + |\mathbf{A}|)^{p-2} |D\bv-\mathbf{A}|^2 \, dx \right)^{\frac{p}{2}} \left( \dashint_{B_{\rho}} (|D\bv|+|\mathbf{A}|)^p \, dx \right)^{\frac{2-p}{2}}\\
        &\hspace{1mm}\leq \frac{c}{\mu_0} \left( \dashint_{B_{\rho}} \varphi''(|D\bv|+|\mathbf{A}|) |D\bv-\mathbf{A}|^2  dx\right)^{\frac{p}{2}} \left( \dashint_{B_{\rho}} (|\bV_{\varphi}(D\bv)|^2 + |\bV_{\varphi}(\mathbf{A})|^2) \, dx \right)^{\frac{2-p}{2}}\\
        &\hspace{1mm}\leq \frac{c}{\mu_0} \left( \dashint_{B_{\rho}} |\bV_{\varphi}(D\bv) -\bV_{\varphi}(\mathbf{A})|^2 \, dx \right)^{\frac{p}{2}}\\
        &\hspace{15mm}\times \left( \dashint_{B_{\rho}} (|\bV_{\varphi}(D\bv)|^2 + |(\bV_{\varphi}(D\bv))_{B_{\rho}}|^2) \, dx \right)^{\frac{2-p}{2}}\\
        &\hspace{1mm} \leq \frac{c}{\mu_0} \left( \dashint_{B_{\rho}} |\bV_{\varphi}(D\bv) - (\bV_{\varphi}(D\bv))_{B_{\rho}}|^2 \, dx \right)^{\frac{p}{2}} \sup_{B_{\rho}} |\bV_{\varphi}(D\bv)|^{2-p}\\
        &\hspace{1mm} \leq \frac{c}{\mu_0} \left(\dfrac{\rho}{R}\right)^{\frac{\alpha_1 p}{2}} \left( \dashint_{B_{R}} |\bV_{\varphi}(D\bv) - (\bV_{\varphi}(D\bv))_{B_{R}}|^2 \, dx \right)^{\frac{p}{2}}
        \left(  \sup_{B_{\rho}} \varphi(|D\bv|)\right)^{\frac{2-p}{2}}\\
        &\hspace{1mm} \leq  \frac{c}{\mu_0} \left(\dfrac{\rho}{R}\right)^{\frac{\alpha_1 p}{2}} \left( \dashint_{B_{R}} (|\bV_{\varphi}(D\bv)|^2 + |(\bV_{\varphi}(D\bv))_{B_{R}}|^2) \, dx \right)^{\frac{p}{2}}\\
        &\hspace{15mm}\times \left( \dashint_{B_R} (a_0 |D\bv|^p + b_0 |D\bv|^p \log(e+|D\bv|)) \, dx \right)^{\frac{2-p}{2}}\\
        &\hspace{1mm} \leq  \frac{c}{\mu_0} \left(\dfrac{\rho}{R}\right)^{\frac{\alpha_1 p}{2}} \dashint_{B_R} (a_0 |D\bv|^p + b_0 |D\bv|^p \log(e+|D\bv|)) \, dx.
    \end{align*}
    Hence \eqref{ineq : domain expansion} holds with $\tilde{\alpha} \coloneq \frac{\alpha_1}{2}$ when $1<p<2$.
\end{proof}

The next theorem shows that the higher integrability of the boundary datum is inherited by the minimizer of the frozen functional, uniformly with respect to $a_0$ and $b_0$.
\begin{theorem}\label{thm:hi-phi}
Let $p>1$, $a_0,b_0\ge0$ with $a_0+b_0>0$, and
$\varphi(t) \coloneq a_0t^p+b_0t^p\log(e+t)$.
Let $B \subset \mr^n$ be a ball, let $\bu \in W^{1,1}(B;\mr^N)$ with
$\varphi(|D\bu|) \in L^{1+\sigma}(B)$ for some $\sigma \in (0,\sigma_0]$, and let
$\bv \in \bu + W^{1,1}_0(B;\mr^N)$ be a minimizer of
$$
\bw \mapsto \int_B \varphi(|D\bw|)\,dx
$$
in this class. Then $\varphi(|D\bv|)\in L^{1+\sigma}(B)$ and
\begin{equation}\label{eq:hi-phi}
\dashint_B\varphi(|D\bv|)^{1+\sigma}\,dx \le c \,\dashint_B\varphi(|D\bu|)^{1+\sigma}\,dx ,
\end{equation}
where $\sigma_0\in(0,1)$ and $c\ge1$ depend only on $n,N,p$, and in particular not on $a_0,b_0$.
\end{theorem}

\begin{proof}
By Remark \ref{cond : conv ft prop}, $\varphi$ satisfies \eqref{cond : mono map} with $c_\varphi \equiv c_\varphi(p)$, independently of $a_0,b_0$. Hence, all constants below depend only on $n,N,c_\varphi$, and in particular are uniform with respect to $a_0,b_0$. We write $\varphi(\mathbf{P})$ for $\varphi(|\mathbf{P}|)$. Let $r_0$ be the radius of $B$. For $y\in\overline B$ and $0<\varrho\le r_0$, the set $B\cap B_\varrho(y)$ contains a ball of radius $\varrho/2$, and hence
$$
|B\cap B_\varrho(y)| \ge 2^{-n}|B_\varrho|.
$$

Fix $x_0\in\overline B$ and $0<r\le \frac{r_0}{5}$. Let $\eta\in C_0^\infty(B_r(x_0))$ satisfy
$$
0 \le \eta \le 1, \qquad \eta \equiv 1\ \text{on }B_{r/2}(x_0),
\qquad |D\eta| \le \frac{4}{r}.
$$
We set $\mathbf{k} \coloneq (\bv-\bu)_{B_r(x_0)}$ if $B_{2r}(x_0)\subset B$ and $\mathbf{k} \coloneq 0$ otherwise, and we define
$$
\bw \coloneq \bv-\eta(\bv-\bu-\mathbf{k}).
$$
If $\mathbf{k}=0$, then $\eta(\bv-\bu) \in W_0^{1,1}(B;\mathbb R^N)$, since $\bv-\bu \in W_0^{1,1}(B;\mathbb R^N)$ and $\eta$ is bounded and Lipschitz. If $B_{2r}(x_0)\subset B$, then $\eta(\bv-\bu-\mathbf{k})$ has compact support in $B_r(x_0) \subset B$. In both cases $\bw-\bv \in W_0^{1,1}(B;\mathbb R^N)$, so $\bw$ is an admissible competitor for $\bv$. Moreover, $\bw = \bv$ outside $B_r(x_0)$. Hence, by the minimality of $\bv$,
$$
\int_{B\cap B_r(x_0)} \varphi(D\bv) \,dx \le \int_{B\cap B_r(x_0)} \varphi(D\bw) \,dx.
$$
Since
$$
D\bw = (1-\eta) D\bv + \eta D\bu-(\bv-\bu-\mathbf{k}) \otimes D\eta,
$$
the structural estimates for $\varphi$ imply
$$
\varphi(D\bw) \le c\,(1-\eta)^p\varphi(D\bv) +c\,\varphi(D\bu) +c\,\varphi\left(\frac{|\bv-\bu-\mathbf{k}|}{r}\right).
$$
Therefore, using that $\eta \equiv 1$ on $B_{r/2}(x_0)$ and applying the standard hole-filling argument together with Lemma \ref{lem : iteration lemma}, we obtain
\begin{equation}\label{eq:cacc-uniform}
\dashint_{B\cap B_{r/2}(x_0)}\varphi(D\bv)\,dx \le c\,  \dashint_{B\cap B_r(x_0)}\varphi(D\bu)\,dx + c\,\dashint_{B\cap B_r(x_0)} \varphi\left(\frac{|\bv-\bu-\mathbf{k}|}{r}\right)\,dx.
\end{equation}

We next estimate the last term in \eqref{eq:cacc-uniform}, including the case where $x_0$ is close to $\partial B$. If $B_{2r}(x_0)\subset B$, then $\mathbf{k}=(\bv-\bu)_{B_r(x_0)}$, and Lemma \ref{rem: 2.2} applied to $\mathbf{f}=\bv-\bu$ gives
$$
\dashint_{B_r(x_0)} \varphi\left(\frac{|\bv-\bu-\mathbf{k}|}{r}\right)\,dx \le  c \left( \dashint_{B_r(x_0)} [\varphi(|D\bv-D\bu|)]^{d_1}\,dx \right)^{\frac{1}{d_1}},
$$
and hence, by the structural inequality for $\varphi$,
$$
\dashint_{B_r(x_0)} \varphi \left(\frac{|\bv-\bu-\mathbf{k}|}{r}\right)\,dx \le c\left( \dashint_{B_r(x_0)} [\varphi(D\bv)]^{d_1}\,dx \right)^{\frac{1}{d_1}}
+ c\,\dashint_{B_r(x_0)}\varphi(D\bu)\,dx.
$$

Suppose now that $B_{2r}(x_0)\not\subset B$, so that $\mathbf{k}=0$. Choose $x_0'\in\partial B$ such that $|x_0-x_0'|<2r$. Then
$$
B\cap B_r(x_0)\subset B\cap B_{3r}(x_0').
$$
Set
$$
\mathcal{U} \coloneq B\cap B_{3r}(x_0')
$$
and define the zero extension of $\bv-\bu$ to $B_{3r}(x_0')$ by
$$
\widetilde{\mathbf{f}}(x) \coloneq
\begin{cases}
\bv(x)-\bu(x), & x\in \mathcal{U},\\
0,&x\in B_{3r}(x_0')\setminus B.
\end{cases}
$$
Since $\bv-\bu$ has zero trace on $\partial B$, we have
$$
\widetilde{\mathbf{f}} \in W^{1,1}(B_{3r}(x_0');\mathbb R^N)
$$
and
$$
D\widetilde{\mathbf{f}}=
\begin{cases}
D\bv-D\bu,&x\in \mathcal{U},\\
0,&x\in B_{3r}(x_0')\setminus B.
\end{cases}
$$
Moreover, since $x_0'\in\partial B$ and $B$ is convex, the set $B_{3r}(x_0')\setminus B$ contains the half-ball cut off by the tangent hyperplane to $\partial B$ at $x_0'$. In particular,
$$
\left|B_{3r}(x_0')\setminus B\right| \ge c|B_{3r}|.
$$
Since $\widetilde{\mathbf{f}}$ vanishes on a fixed proportion of $B_{3r}(x_0')$, Lemma \ref{rem: 2.2} yields the following Sobolev--Poincar\'e inequality without the mean value:
$$
\dashint_{B_{3r}(x_0')} \varphi\left(\frac{|\widetilde{\mathbf{f}}|}{r}\right)\,dx \le 
c \left( \dashint_{B_{3r}(x_0')} [\varphi(|D\widetilde{\mathbf{f}}|)]^{d_1}\,dx \right)^{\frac{1}{d_1}}.
$$
Consequently,
$$
\dashint_{B\cap B_{3r}(x_0')} \varphi\left(\frac{|\bv-\bu|}{r}\right)\,dx \le c\left( \dashint_{B\cap B_{3r}(x_0')} [\varphi(|D\bv-D\bu|)]^{d_1}\,dx \right)^{\frac{1}{d_1}}.
$$
Using
$$
\varphi(|D\bv-D\bu|) \le c \left(\varphi(D\bv)+\varphi(D\bu)\right)
$$
and the fact that $d_1<1$, we further obtain
$$
\begin{aligned}
\dashint_{B\cap B_{3r}(x_0')} \varphi \left( \frac{|\bv-\bu|}{r} \right)\,dx &\le c\left(\dashint_{B\cap B_{3r}(x_0')} [\varphi(D\bv)]^{d_1}\,dx\right)^{\frac{1}{d_1}}\\
&\quad + c\left( \dashint_{B\cap B_{3r}(x_0')} [\varphi(D\bu)]^{d_1}\,dx \right)^{\frac{1}{d_1}} \\
&\le c \left( \dashint_{B\cap B_{3r}(x_0')} [\varphi(D\bv)]^{d_1}\,dx \right)^{\frac{1}{d_1}}\\
&\quad + c\,\dashint_{B\cap B_{3r}(x_0')}\varphi(D\bu)\,dx.
\end{aligned}
$$
Since $B\cap B_{3r}(x_0') \subset B\cap B_{5r}(x_0)$ and $5r \le r_0$, the lower bound for $|B\cap B_\varrho(y)|$ allows us to replace the averages over $B\cap B_r(x_0)$ and $B\cap B_{3r}(x_0')$ by averages over $B\cap B_{5r}(x_0)$. Combining the interior and boundary cases in this way, we obtain, for every $x_0\in\overline B$ and $0 < r \le \frac{r_0}{5}$,
\begin{equation}\label{eq:rh-uniform}
\dashint_{B\cap B_{r/2}(x_0)}\varphi(D\bv)\,dx \le c\left( \dashint_{B\cap B_{5r}(x_0)} [\varphi(D\bv)]^{d_1}\,dx \right)^{\frac{1}{d_1}} + c \, \dashint_{B\cap B_{5r}(x_0)}\varphi(D\bu)\,dx.
\end{equation}

If we extend $\varphi(D\bv)$ and $\varphi(D\bu)$ by zero outside $B$, then \eqref{eq:rh-uniform} becomes a reverse H\"older inequality with a right-hand side. Hence, by the corresponding version of Gehring's lemma, there exists
$$
\sigma_0=\sigma_0(n,N,c_\varphi)\in(0,1)
$$
such that, for every $\sigma\in(0,\sigma_0]$,
$$
\dashint_B\varphi(D\bv)^{1+\sigma}\,dx \le c\left( \,\dashint_B\varphi(D\bv)\,dx\right)^{1+\sigma} + c\, \dashint_B\varphi(D\bu)^{1+\sigma}\,dx.
$$
Finally, by the minimality of $\bv$,
$$
\int_B\varphi(D\bv)\,dx \le \int_B\varphi(D\bu)\,dx,
$$
and Jensen's inequality gives
$$
\dashint_B\varphi(D\bu)\,dx \le \left( \,\dashint_B\varphi(D\bu)^{1+\sigma}\,dx \right)^{\frac{1}{1+\sigma}}.
$$
Combining these inequalities with the preceding estimate yields \eqref{eq:hi-phi}.
\end{proof}

\section{\bf Higher integrability}\label{section 4}
In this section, we prove the higher integrability of local minimizers. Let $\Omega \subset \mr^n$ be a bounded domain. We consider the integral functional
\begin{equation}\label{def : functional 4}
    W^{1,1}(\Omega) \ni \bw \mapsto \mathcal{F}(\bw, \Omega) \coloneq \int_{\Omega} F(x,\bw,D\bw) \, dx
\end{equation}
where the energy density $F:\Omega \times \mr^N \times \mr^{N \times n} \rightarrow \mr$ is a Carath\'eodory function satisfying the following conditions:
\begin{enumerate}[label=\textbf{(H\arabic*)}, leftmargin=1.5cm, itemsep=0.5em]
    \item There exist constants $0 < \nu \leq 1 \leq L$ such that
        \begin{equation}\label{cond : H and F}
        \nu H(x,\mathbf{P}) \leq F(x,\bv,\mathbf{P}) \leq L H(x,\mathbf{P})
        \end{equation}
        for all $x\in \Omega$, $\bv \in \mr^N$, and $\mathbf{P} \in \mr^{N \times n}$, where the model integrand is given by
        \begin{equation}\label{def : function H}
        H(x,\mathbf{P}) \coloneq a(x)|\mathbf{P}|^p + b(x) |\mathbf{P}|^p \log(e+|\mathbf{P}|).
        \end{equation}
        By abuse of notation, we use the same expression $H(x,\mathbf{z})$ for vectors $\mathbf{z} \in \mr^N$ and $H(x,t)$ for scalars $t \geq 0$.
    \item The modulating coefficient $b(\cdot)$ is $\log$-H\"older continuous, that is, its modulus of continuity $\omega_b(\cdot)$ from \eqref{cond : modulus of b} satisfies
    \begin{equation}\label{cond : Logarithmic}
        \limsup_{r \rightarrow 0} \omega_b(r) \log\left( \dfrac{1}{r} \right) < \infty,
    \end{equation}
    or equivalently, since $b$ is bounded,
    \begin{equation}\label{cond : Logarithmic 2}
        \omega_b(r) \log\left( \dfrac{1}{r} \right) \leq \tilde{L} \; \text{ for every } r \leq 1
    \end{equation}
    for some constant $\tilde{L} \geq 0$. Without loss of generality, we may assume that $\omega_a(\cdot)$ and $\omega_b(\cdot)$ are concave.
\end{enumerate}

The following theorem gives the higher integrability of local minimizers of $\mathcal{F}$.
\begin{theorem}\label{thm : Gehring's theory}
    Let $\bu \in W^{1,p}(\Omega;\mr^N)$ be a local minimizer of the functional $\mathcal{F}$ defined in \eqref{def : functional 4} under the assumptions \textbf{(H1)} and \textbf{(H2)}. By \eqref{cond : omega_a} and \eqref{cond : Logarithmic}, there exists $R_* \in (0,\frac{1}{e}]$, depending only on $\mu$, $\omega_a$, and $\omega_b$, such that
    \begin{equation}\label{def : R_*}
        4\left(\omega_a(R)+\omega_b(R)\right)<\mu \quad \text{for all } 0<R\le R_*.
    \end{equation}
    Then there exists an exponent $\delta_g > 0$, depending only on $n$, $N$, $p$, $\nu$, $L$, $\mu$, $\|a\|_{L^\infty(\Omega)}$, $\| D\bu \|_{L^p}$, and $\tilde{L}$, such that
    \begin{equation}
        H(x,D\bu) \in L_{\operatorname{loc}}^{1+\delta_{g}}(\Omega).
    \end{equation}
    More precisely, the local reverse H\"older inequality
    \begin{equation}\label{ineq : reverse holder}
        \left( \dashint_{B_{R/2}} [H(x,D\bu)]^{1+\delta_g} \, dx \right)^{\frac{1}{1+\delta_g}} \leq c\, \dashint_{B_R} H(x, D\bu) \, dx
    \end{equation} 
    holds for every ball $B_R \subset \Omega$ with $R \le R_*$ and for a constant $c$ depending only on $n$, $N$, $p$, $\nu$, $L$, $\mu$, $\|a\|_{L^\infty(\Omega)}$, $\| D\bu \|_{L^p}$, and $\tilde{L}$. In particular, if $p > \dfrac{n}{1+\delta_g}$, then $\bu$ is locally H\"older continuous.
\end{theorem}

We prepare several lemmas for the proof of Theorem \ref{thm : Gehring's theory}. The first one is the following quasi-convexity property of $H$.
\begin{lemma}
Let $C_p \coloneq 2^p$. Then
\begin{equation}\label{lem:quasi-convexity}
H(x, \mathbf{A}+\mathbf{B}) \leq C_p \big( H(x, \mathbf{A}) + H(x, \mathbf{B}) \big)
\end{equation}
for all $x \in \Omega$ and $\mathbf{A}, \mathbf{B} \in \mathbb{R}^{N \times n}$.
\end{lemma}

\begin{proof}
    Since $t\mapsto t^p$ is convex on $[0,\infty)$ for $p\ge1$, Jensen's inequality gives
    $$
    |\mathbf{A}+\mathbf{B}|^p \leq (|\mathbf{A}|+|\mathbf{B}|)^p \leq 2^{p-1} (|\mathbf{A}|^p + |\mathbf{B}|^p).    
    $$
    On the other hand, since $\Phi(t) \coloneq t^p \log(e+t)$ is convex on $[0, \infty)$ and $\log(e+2t) \leq 2\log(e+t)$, Jensen's inequality yields
    \begin{align*}
    (|\mathbf{A}|+|\mathbf{B}|)^p \log(e+|\mathbf{A}|+|\mathbf{B}|) 
    &= 2^p \left( \frac{|\mathbf{A}|+|\mathbf{B}|}{2} \right)^p \log\left(e + 2 \cdot \frac{|\mathbf{A}|+|\mathbf{B}|}{2}\right) \\
    &\leq 2^{p+1} \Phi\left( \frac{|\mathbf{A}|+|\mathbf{B}|}{2} \right) \\
    &\leq 2^p \Big( |\mathbf{A}|^p \log(e+|\mathbf{A}|) + |\mathbf{B}|^p \log(e+|\mathbf{B}|) \Big).
    \end{align*}
    Multiplying these inequalities by $a(x) \ge 0$ and $b(x) \ge 0$, respectively, and adding them, we obtain
    \begin{align*}
    &H(x, \mathbf{A}+\mathbf{B}) = a(x)|\mathbf{A}+\mathbf{B}|^p + b(x)|\mathbf{A}+\mathbf{B}|^p \log(e+|\mathbf{A}+\mathbf{B}|) \\
    &\hspace{14mm} \leq 2^{p-1} a(x) (|\mathbf{A}|^p + |\mathbf{B}|^p) + 2^p b(x) \Big( |\mathbf{A}|^p \log(e+|\mathbf{A}|) + |\mathbf{B}|^p \log(e+|\mathbf{B}|) \Big) \\
    &\hspace{14mm}\leq 2^p \big( H(x, \mathbf{A}) + H(x, \mathbf{B}) \big),
    \end{align*}
    which completes the proof.
    \end{proof}

The second one is an elementary scaling property of $H$.
\begin{lemma}
    Let $0 \leq \lambda \leq 1$. Then
    \begin{equation}\label{lem:H-cutoff}
        H(x,\lambda \mathbf{A}) \leq \lambda^pH(x,\mathbf{A})    
    \end{equation}
    for every $\mathbf{A} \in \mathbb{R}^{N\times n}$.
\end{lemma}

\begin{proof}
    Since $0 \leq \lambda \leq 1$,
    $$
    \log(e+\lambda|\mathbf{A}|) \leq \log(e+|\mathbf{A}|),
    $$
    and therefore
    $$
    \begin{aligned}
    H(x,\lambda \mathbf{A}) &= a(x)|\lambda \mathbf{A}|^p + b(x)|\lambda \mathbf{A}|^p \log(e+\lambda|\mathbf{A}|)\\
    &= \lambda^pa(x)|\mathbf{A}|^p + \lambda^pb(x)|\mathbf{A}|^p \log(e+\lambda|\mathbf{A}|)\\
    &\leq \lambda^p \left( a(x)|\mathbf{A}|^p + b(x)|\mathbf{A}|^p \log(e+|\mathbf{A}|) \right)\\
    &= \lambda^pH(x,\mathbf{A}).
    \end{aligned}
    $$
\end{proof}

Next, we prove a Caccioppoli-type inequality.
\begin{lemma}
    Let $\bu \in W^{1,p}(\Omega; \mr^{N})$ be a local minimizer of the functional $\mathcal{F}$ defined in \eqref{def : functional 4}, and let $B_R \subset \Omega$ be a ball. Then, for every $\mathbf{k} \in \mr^N$,
    \begin{equation}\label{ineq : Caccioppoli}
        \dashint_{B_{R/2}} H(x, D\bu)\, dx \leq c \, \dashint_{B_R} H \left(x, \dfrac{\bu-\mathbf{k}}{R}\right)\, dx
    \end{equation}
    for a constant $c$ depending only on $n$, $N$, $p$, $\nu$, and $L$.
\end{lemma}
\begin{proof}
    Fix a constant vector $\mathbf{k}=(k^1,\cdots,k^N)\in\mathbb{R}^N$ and radii $\frac{R}{2}\le r<s\le R$. We choose a cut-off function $\eta\in C_0^\infty(B_s)$ such that $0\le\eta\le1$, $\eta\equiv1$ on $B_r$, and $|D\eta|\le\frac{2}{s-r}$, and we define the comparison function $\mathbf{w}=(w^1, \cdots, w^N)$ by
    $$w^i(x) \coloneq u^i(x) - \eta(x)(u^i(x) - k^i) = (1 - \eta(x))u^i(x) + \eta(x)k^i, \quad i=1, \cdots, N.$$
    Since $\eta = 0$ on $\partial B_s$, we have $\mathbf{w} - \bu \in W_0^{1,p}(B_s; \mathbb{R}^N)$.
    Differentiating componentwise, we obtain
    $$D_\alpha w^i = (1-\eta)D_\alpha u^i - (u^i-k^i)D_\alpha\eta, \quad \alpha=1,\cdots,n,\; i = 1, \cdots, N.$$
    By the local minimality of $\bu$ and condition \eqref{cond : H and F},
    \begin{equation}\label{4.11}
    \begin{aligned}
        \nu \int_{B_s} H(x, D\bu) \, dx &\le \int_{B_s} F(x, \bu, D\bu) \, dx \\
        &\le \int_{B_s} F(x, \mathbf{w}, D\mathbf{w}) \, dx \le L \int_{B_s} H(x, D\mathbf{w}) \, dx.
    \end{aligned}
    \end{equation}
    By \eqref{lem:quasi-convexity} and \eqref{lem:H-cutoff} with $\lambda=1-\eta$, we have
    $$
    H(x,D\mathbf{w}) \leq C_p (1-\eta)^p H(x,D\bu) + C_p \, H\!\left( x, (\bu-\mathbf{k}) \otimes D\eta \right).
    $$
    For $\lambda \geq 1$ and $t \geq 0$, \eqref{cond : log} gives $\log(e+\lambda t) \leq \lambda \log(e+t)$, and hence
    $$
    H(x,\lambda t) \leq \lambda^{p+1} H(x,t).
    $$
    Since $|(\bu-\mathbf{k}) \otimes D\eta| \leq \lambda \frac{|\bu-\mathbf{k}|}{R}$ with $\lambda \coloneq \frac{2R}{s-r} \geq 1$, it follows that
    $$
    H\!\left(x,(\bu-\mathbf{k}) \otimes D\eta\right) \leq \left(\frac{2R}{s-r}\right)^{p+1} H\!\left(x,\frac{\bu-\mathbf{k}}{R}\right).
    $$
    Substituting these estimates into \eqref{4.11} and using that $(1-\eta)^p=0$ in $B_r$, we obtain
    $$
    \int_{B_r}H(x,D\bu)\,dx \leq c_* \int_{B_s\setminus B_r}H(x,D\bu)\,dx + c \left(\frac{R}{s-r}\right)^{p+1} \int_{B_R} H\!\left(x,\frac{\bu-\mathbf{k}}{R}\right)\,dx,
    $$
    where $c_*=c_*(p,\nu,L)$ and $c=c(p,\nu,L)$.
    By filling the hole, that is, adding
    $$ c_* \int_{B_r}H(x,D\bu)\,dx $$
    to both sides, we obtain
    $$
    \int_{B_r}H(x,D\bu)\,dx \leq \theta \int_{B_s}H(x,D\bu)\,dx + \frac{c\,R^{p+1}}{(s-r)^{p+1}} \int_{B_R} H\!\left(x,\frac{\bu-\mathbf{k}}{R}\right)\,dx,
    $$
    where $ \theta=\frac{c_*}{c_*+1}\in(0,1)$.
    We set
    $$
    h(t)=\int_{B_t}H(x,D\bu)\,dx, \qquad \frac{R}{2} \leq t \leq R.
    $$
    Since $h$ is bounded, Lemma \ref{lem : iteration lemma} with $\kappa=p+1$ yields
    $$
    \int_{B_{R/2}}H(x,D\bu)\,dx \leq c \int_{B_R} H\!\left(x,\frac{\bu-\mathbf{k}}{R}\right)\,dx,
    $$
    where $c=c(p,\nu,L)$.
    Finally, dividing both sides by $|B_{R/2}|$ and using
    $|B_R|=2^n|B_{R/2}|$, we obtain \eqref{ineq : Caccioppoli}.
\end{proof}

We also use the Poincar\'e inequality for vector-valued functions.
\begin{lemma}\label{poincare vector}
Let $B_R\subset\mathbb{R}^n$ be a ball and let $\bu \in W^{1,p}(B_R;\mr^N)$.
Then
\begin{equation}\label{vector poincare}
    \left(\, \dashint_{B_R} \left|\bu-(\bu)_{B_R}\right|^p\,dx \right)^{\frac{1}{p}}
    \leq
    cR \left(\, \dashint_{B_R}|D\bu|^p\,dx \right)^{\frac{1}{p}},
\end{equation}
where $c=c(n,N,p)$.
\end{lemma}
\begin{proof}
    Write $\bu=(u_1,\cdots,u_N)$. Applying the scalar-valued Poincar\'e inequality to each component $u_\alpha$, $\alpha=1,\cdots,N$, we obtain
    $$
    \sum_{\alpha=1}^N\int_{B_R}|u_\alpha-(u_\alpha)_{B_R}|^p\,dx
    \leq cR^p \, \sum_{\alpha=1}^N\int_{B_R}|Du_\alpha|^p\,dx.
    $$
    Since   
    $$
    |\bu-(\bu)_{B_R}|^p = \Big(\sum_{\alpha=1}^N |u_\alpha-(u_\alpha)_{B_R}|^2 \Big)^{\frac {p}{2}},   
    $$
    and
    $$
    |D\bu|^p = \Big( \sum_{\alpha=1}^N |Du_\alpha|^2 \Big)^{\frac {p}{2}},
    $$
    the equivalence of norms in the finite-dimensional space $\mathbb{R}^N$ yields
    $$
    \int_{B_R} |\bu-(\bu)_{B_R}|^p \, dx \leq cR^p \int_{B_R} |D\bu|^p \, dx.
    $$
    Dividing by $|B_R|$ and taking the $p$-th root gives
    $$
    \left(\, \dashint_{B_R} |\bu-(\bu)_{B_R}|^p \, dx \right)^{\frac{1}{p}} \leq cR \, \left(\,\dashint_{B_R}|D\bu|^p \, dx \right)^{\frac{1}{p}}.
    $$
\end{proof}
\medskip
\begin{proof}[Proof of Theorem \ref{thm : Gehring's theory}]
    By Lemma \ref{Lem : Gehring lemma}, it suffices to show that there exist an exponent $d \in (0, 1)$, depending only on $n$, $N$, and $p$, and a constant $c$, depending only on $n$, $N$, $p$, $\nu$, $L$, $\mu$, $\|a\|_{L^\infty(\Omega)}$, $\|D\bu\|_{L^p}$, and $\tilde{L}$, such that the reverse H\"older inequality
    \begin{equation}\label{ineq : rev}
        \dashint_{B_{R/2}} H(x, D\bu) \, dx \le c \left(\, \dashint_{B_R} [H(x, D\bu)]^d \, dx \right)^{\frac{1}{d}}
    \end{equation}
    holds for every ball $B_R \subset \Omega$ with $R \le R_*$. Indeed, Lemma \ref{Lem : Gehring lemma} applied with $f = H(x, D\bu)$, $\tilde{c} = c$, and $R_0 = R_*$ then gives an exponent $\delta_g > 0$ such that $H(x, D\bu) \in L_{\operatorname{loc}}^{1+\delta_g}(\Omega)$ and
    $$
    \dashint_{B_{R/2}} [H(x, D\bu)]^{1+\delta_g} \, dx \le C \left(\, \dashint_{B_R} H(x, D\bu) \, dx \right)^{1+\delta_g}
    $$
    for every ball $B_R \subset \Omega$ with $R \le R_*$, which is \eqref{ineq : reverse holder}. The last assertion of the theorem then follows from Morrey's embedding theorem, since $\mu|D\bu|^p \le H(x,D\bu)$ by \eqref{cond : a,b}.

    To prove \eqref{ineq : rev}, we fix a ball $B_R \subset \Omega$ with $R \leq R_*$ and set
    $$
    b_i(R) \coloneq \inf_{B_R}b \quad \text{and} \quad a_i(R) \coloneq \inf_{B_R} a.
    $$
    We distinguish two cases according to whether
    \begin{equation}\label{cond : b_i}
        b_i(R)\leq 2\omega_b(R)    
    \end{equation}
    or
    \begin{equation}\label{cond : b_i 2}
        b_i(R)>2\omega_b(R).
    \end{equation}
    
    We first consider the case \eqref{cond : b_i}. Then, for every $x \in B_R$,
    $$
    b(x) = [b(x) - b_i(R)] + b_i(R) \leq \omega_b(2R) + b_i(R) \leq 2 \omega_b (R) + b_i(R) \leq 4 \omega_b(R)
    $$ 
    where we have used the concavity of $\omega_b(\cdot)$ to get $\omega_b(2R )\le 2\omega_b(R)$.
    By the vector-valued Poincar\'e inequality \eqref{vector poincare}, we have
    $$
    \left( \,\dashint_{B_R} |\bu - (\bu)_{B_R}|^p \, dx \right)^{\frac{1}{p}} \le c R \left(\, \dashint_{B_R} |D\bu|^p \, dx \right)^{\frac{1}{p}},
    $$
    which, after taking the $p$-th power and using $|B_R| = \alpha_n R^n$, yields the scale-invariant estimate
    $$\frac{(|\bu - (\bu)_{B_R}|^p)_{B_R}}{R^p} \le \frac{c^p}{\alpha_n R^n} \int_{B_R} |D\bu|^p \, dx = \frac{c_1}{R^n} \|D\bu\|_{L^p(B_R)}^p.$$
    Using the elementary inequality $\log(e + t) \le c [1 + \log(e + t^p)]$ with $t = \frac{|\bu - (\bu)_{B_R}|}{R}$ together with \eqref{cond : log}, we deduce that
    \begin{equation}\label{ineq : 4.13_1}
        \begin{aligned}
        &\log \left( e + \left| \frac{\bu - (\bu)_{B_R}}{R} \right| \right) \le c \left[ 1 + \log \left( e + \left| \frac{\bu - (\bu)_{B_R}}{R} \right|^p \right) \right] \\ 
        &\hspace{3mm}\le c + c \log \left( e + \frac{|\bu - (\bu)_{B_R}|^p}{(|\bu - (\bu)_{B_R}|^p)_{B_R}} \right) + c \log \left( e + \frac{(|\bu - (\bu)_{B_R}|^p)_{B_R}}{R^p} \right) \\ 
        &\hspace{3mm}\le c + c \log \left( e + \frac{|\bu - (\bu)_{B_R}|^p}{(|\bu - (\bu)_{B_R}|^p)_{B_R}} \right) + c \log \left( e + \frac{\|D\bu\|_{L^p(B_R)}^p}{R^n} \right) \\ 
        &\hspace{3mm}\le c \log \left( e + \frac{|\bu - (\bu)_{B_R}|^p}{(|\bu - (\bu)_{B_R}|^p)_{B_R}} \right) + c \log \left( \frac{1}{R} \right), 
    \end{aligned}    
    \end{equation}
    where $c \equiv c(n, N, p, \|D\bu\|_{L^p})$ and, in the last step, the constant terms are absorbed by using $\log\left(\frac{1}{R}\right) \ge 1$, which holds since $R \le \frac{1}{e}$.
    Under \eqref{cond : b_i}, we distinguish the two sub-cases $a_i(R) \le 2\omega_a(R)$ and $a_i(R) > 2\omega_a(R)$.

    Suppose first that $a_i(R)\le 2\omega_a(R)$. Then, for every $x\in B_R$,
    $$
    a(x)\le a_i(R)+\omega_a(2R) \le 4\omega_a(R).
    $$
    Since also $b(x) \le 4\omega_b(R)$, we obtain
    $$
    a(x)+b(x) \le 4\left(\omega_a(R)+\omega_b(R)\right) < \mu,
    $$
    where the last inequality follows from \eqref{def : R_*} and $R \le R_*$.
    This contradicts \eqref{cond : a,b}, so this sub-case cannot occur.
    Hence $a_i(R) > 2\omega_a(R)$, and therefore
    $$ a(x) = [a(x) - a_i(R)] + a_i(R) \le \omega_a(2R) + a_i(R) \le 2\omega_a(R) + a_i(R) \le 2a_i(R)$$
    for all $x \in B_R$. Using this, $b(x) \le 4\omega_b(R)$, \eqref{ineq : Caccioppoli} with $\mathbf{k} = (\bu)_{B_R}$, and \eqref{ineq : 4.13_1}, we get
    $$
    \begin{aligned}
        \dashint_{B_{R/2}} H(x,D\bu)\, dx &\le c \,\dashint_{B_R} H\left(x, \frac{\bu - (\bu)_{B_R}}{R}\right) dx \\
        &\le c \, \dashint_{B_R} a_i(R)\left| \frac{\bu - (\bu)_{B_R}}{R} \right|^p dx \\
        &\quad + {c\,\omega_b(R)} \dashint_{B_R} \left| \frac{\bu - (\bu)_{B_R}}{R} \right|^p \log \left( e + \left| \frac{\bu - (\bu)_{B_R}}{R} \right| \right) dx \\
        &\le c \left( a_i(R) + {\omega_b(R)} \log\left(\frac{1}{R}\right) \right) \dashint_{B_R} \left| \frac{\bu - (\bu)_{B_R}}{R} \right|^p dx \\
        &\quad + c\,\omega_b(R) \dashint_{B_R} \left| \frac{\bu - (\bu)_{B_R}}{R} \right|^p \log \left( e + \frac{|\bu - (\bu)_{B_R}|^p}{(|\bu - (\bu)_{B_R}|^p)_{B_R}} \right) dx.
    \end{aligned}
    $$

    Applying \eqref{L log L} with $f \equiv \frac{|\bu -(\bu)_{B_R}|^p}{R^{p}}$, $\gamma = 1$, and the exponent $\frac{q}{p}$ for some $q > p$ to the last term, we obtain
    $$
    \begin{aligned}
        &\dashint_{B_{R/2}} H(x,D\bu)\, dx\\
        & \le c \left( a_i(R) + \omega_b(R) \log\left(\frac{1}{R}\right) \right)\\
        &\qquad \times \left[ \dashint_{B_R} \left| \frac{\bu - (\bu)_{B_R}}{R} \right|^p dx + \left( \dashint_{B_R} \left| \frac{\bu - (\bu)_{B_R}}{R} \right|^q dx \right)^{\frac{p}{q}} \right]
    \end{aligned}
    $$
    where $c$ depends only on $n$, $N$, $p$, $q$, $\nu$, $L$, and $\|D\bu\|_{L^p}$. By the Sobolev--Poincar\'e inequality, we can choose $q > p$ and $q_{*} < p$, depending only on $n$ and $p$, such that
    $$
    \left(\, \dashint_{B_R} \left| \frac{\bu - (\bu)_{B_R}}{R} \right|^q dx \right)^{\frac{p}{q}} \le c \left(\, \dashint_{B_R} |D\bu|^{q_*} dx \right)^{\frac{p}{q_*}}
    $$
    and, by H\"older's inequality,
    $$
    \dashint_{B_R} \left| \frac{\bu - (\bu)_{B_R}}{R} \right|^p dx \le c \left( \,\dashint_{B_R} |D\bu|^{q_*} dx \right)^{\frac{p}{q_*}}.
    $$
    Since $a$ is bounded and \eqref{cond : Logarithmic 2} holds, we have
    $$
    c\left(a_i(R) + \omega_b(R)\log\left(\frac{1}{R}\right)\right) \le C.
    $$
    Combining these estimates with the previous inequality, we obtain
    $$
    \dashint_{B_{R/2}} H(x,D\bu)\, dx \le c \left(\, \dashint_{B_R} |D\bu|^{q_*} dx \right)^{\frac{p}{q_*}}.
    $$
    Since \eqref{cond : a,b} implies $\mu|D\bu|^p \le H(x,D\bu)$, we conclude the reverse H\"older inequality
    $$
    \dashint_{B_{R/2}} H(x,D\bu)\, dx \le c \left(\, \dashint_{B_R} [H(x,D\bu)]^{\frac{q_*}{p}} dx \right)^{\frac{p}{q_*}}
    $$
    for some $q_* < p$ and some constant $c$ depending only on $n$, $N$, $p$, $\nu$, $L$, $\mu$, $\|a\|_{L^\infty(\Omega)}$, $\|D\bu\|_{L^p}$, and $\tilde{L}$.

    We now consider the case \eqref{cond : b_i 2}. For every $x \in B_R$, we have
    $$
    b(x) = [b(x)-b_i(R)] + b_i(R) \le \omega_b(2R) + b_i(R) \le 2 \omega_b(R) + b_i(R) \le 2 b_i(R).
    $$
    Moreover, $a(x) \le 2a_i(R)$ if $a_i(R) > 2\omega_a(R)$, and $a(x) \le 4\omega_a(R)$ if $a_i(R) \leq 2\omega_a(R)$. In both cases, we have 
    $$
    a(x) \le \max \{2a_i(R), 4 \omega_a(R)\} \eqcolon a_m(R).
    $$
    Therefore, \eqref{ineq : Caccioppoli} with $\mathbf{k} = (\bu)_{B_R}$ gives
    \begin{align*}
        &\dashint_{B_{R/2}} H(x, D\bu) \, dx \le c\, \dashint_{B_R} H\left(x, \frac{\bu - (\bu)_{B_R}}{R}\right) dx\\
        &\hspace{1.5cm}\le c\, \dashint_{B_R}\left| \dfrac{\bu-(\bu)_{B_R}}{R} \right|^p \left( a_m(R) + b_i(R) \log\left(  e + \left| \dfrac{\bu-(\bu)_{B_R}}{R} \right| \right) \right) dx. 
    \end{align*}

    We now apply the Sobolev--Poincar\'e inequality \eqref{ienq : Sobolev} to the function $\varphi$ in \eqref{rem : 2.1}, which yields
    \begin{align*}
        &\dashint_{B_R} \left| \frac{\bu - (\bu)_{B_R}}{R} \right|^p \left( a_0 + b_0 \log\left( e + \left| \frac{\bu - (\bu)_{B_R}}{R} \right| \right) \right) dx \\
        &\hspace{4cm}\le c \left( \,\dashint_{B_R} |D\bu|^{p d_1} \left( a_0 + b_0 \log(e + |D\bu|) \right)^{d_1} dx \right)^{\frac{1}{d_1}},
    \end{align*}
    with $d_1 \equiv d_1(n,N,p) \in (0,1)$ and $c$ depending only on $n$, $N$, and $p$. Applying this inequality with $a_0 = a_m(R)$ and $b_0 = b_i(R)$, we deduce that
    $$
    \begin{aligned} 
        &\dashint_{B_{R/2}} H(x, D\bu) \, dx \\
        &\hspace{1.3cm}\le c\, \dashint_{B_R} \left| \frac{\bu - (\bu)_{B_R}}{R} \right|^p \left( a_m(R) + b_i(R) \log \left( e + \left| \frac{\bu - (\bu)_{B_R}}{R} \right| \right) \right) dx \\ 
        &\hspace{1.3cm}\le c \left(\, \dashint_{B_R} |D\bu|^{p d_1} \left( a_m(R) + b_i(R) \log(e + |D\bu|) \right)^{d_1} dx \right)^{\frac{1}{d_1}}.
    \end{aligned}
    $$
    It remains to estimate the integrand on the right-hand side. Since $a(\cdot)$ is bounded, there exists a constant $C>0$, independent of $R$, such that $a_m(R)\leq C$. We distinguish two cases according to the size of $b_i(R)$.
    If $b_i(R)\geq\frac{\mu}{4}$, then
    $$
    a_m(R) \leq C \leq \frac{4C}{\mu}b_i(R).
    $$
    Since $\log(e+|D\bu|)\geq 1$, it follows that
    $$
    \begin{aligned}
    a_m(R)+b_i(R)\log(e+|D\bu|) &\leq \frac{4C}{\mu}b_i(R) +b_i(R)\log(e+|D\bu|)\\
    &\leq \left(\frac{4C}{\mu}+1\right) b_i(R)\log(e+|D\bu|).
    \end{aligned}  
    $$
    Since $b_i(R)\leq b(x)$ for every $x\in B_R$, it follows that
    $$
    b_i(R)\log(e+|D\bu|) \leq b(x)\log(e+|D\bu|).
    $$
    Therefore,
    $$
    \begin{aligned}
    &|D\bu|^{pd_1} \left(a_m(R)+b_i(R)\log(e+|D\bu|)\right)^{d_1}\\
    &\qquad\leq c \left( b(x)|D\bu|^p \log(e+|D\bu|) \right)^{d_1}\\
    &\qquad\leq c[H(x,D\bu)]^{d_1}.
    \end{aligned}
    $$  
    Consequently,
    $$
    \dashint_{B_{R/2}} H(x,D\bu)\,dx \leq c\left(\, \dashint_{B_R}[H(x,D\bu)]^{d_1}\,dx \right)^{\frac{1}{d_1}}.
    $$

    If $b_i(R)<\frac{\mu}{4}$, then, since $b_i(R)>2\omega_b(R)$, we have for every $x\in B_R$
    $$
    b(x) \leq b_i(R)+\omega_b(2R) \leq b_i(R)+2\omega_b(R) < 2b_i(R) < \frac{\mu}{2}.
    $$
    Hence \eqref{cond : a,b} gives $a(x) \geq \mu-b(x) > \frac{\mu}{2}$ in $B_R$. Consequently,
    $$
    H(x,D\bu) = |D\bu|^p \left( a(x)+b(x)\log(e+|D\bu|) \right) \geq \frac{\mu}{2}|D\bu|^p,
    $$
    and therefore
    $$
    |D\bu|^p \leq \frac{2}{\mu}H(x,D\bu).
    $$
    Since $a_m(R)\leq C$, we have
    $$
    a_m(R)|D\bu|^p \leq \frac{2C}{\mu}H(x,D\bu).
    $$
    On the other hand, since $b_i(R)\leq b(x)$,
    $$
    b_i(R)|D\bu|^p\log(e+|D\bu|) \leq b(x)|D\bu|^p\log(e+|D\bu|) \leq H(x,D\bu).
    $$
    It follows that
    $$
    \begin{aligned}
    &|D\bu|^{pd_1} \left( a_m(R)+b_i(R)\log(e+|D\bu|) \right)^{d_1}\\
    &\qquad = \left[\, |D\bu|^p \left( a_m(R)+b_i(R)\log(e+|D\bu|) \right) \,\right]^{d_1}\\
    &\qquad\leq c[H(x,D\bu)]^{d_1}.
    \end{aligned}
    $$
    Consequently,
    $$
    \dashint_{B_{R/2}} H(x,D\bu)\,dx \leq c\left( \, \dashint_{B_R}[ H(x,D\bu)]^{d_1}\,dx \right)^{\frac{1}{d_1}}.
    $$
    Combining all cases and using H\"older's inequality, we obtain \eqref{ineq : rev} with $d \coloneq \max\{d_1, \frac{q_*}{p}\} \in (0,1)$ and a constant $c$ depending only on $n$, $N$, $p$, $\nu$, $L$, $\mu$, $\|a\|_{L^\infty(\Omega)}$, $\|D\bu\|_{L^p}$, and $\tilde{L}$. This completes the proof.
\end{proof}

\section{\bf Proof of the main theorem}\label{section 5}
In this section, we prove the two assertions of Theorem \ref{thm : main theorem}.
We first show that $\ell = 0$ in \eqref{cond : ell} implies $\bu \in C_{\operatorname{loc}}^{0,\beta}(\Omega;\mr^N)$ for every $\beta \in (0,1)$. We then prove the H\"older continuity of the gradient $D\bu$, assuming that $a(\cdot)$ and $b(\cdot)$ are H\"older continuous. We apply the results of Section \ref{section 4} with $F(x,\bv,\mathbf{P}) = H(x,\mathbf{P})$, so that $\nu=L=1$. Throughout the remainder of this section, $\tilde{L}$ denotes a finite constant satisfying \eqref{cond : Logarithmic 2}, and $R_* \in (0,\frac{1}{e}]$ is the radius determined by \eqref{def : R_*}. We denote by $\texttt{data}$ the collection of the quantities $n$, $N$, $p$, $\mu$, $\|a\|_{L^\infty(\Omega)}$, $\|b\|_{L^\infty(\Omega)}$, $\|H(\cdot,D\bu)\|_{L^1(\Omega)}$, and $\tilde{L}$. Since $\mu|D\bu|^p \leq H(x,D\bu)$ by \eqref{cond : a,b}, the quantity $\|D\bu\|_{L^p(\Omega)}$ is controlled by $\texttt{data}$.

In the lemma below, we estimate the difference between the energy $\mathcal{P}_{\log}(\bu, B_{R/2})$ of a minimizer of $\mathcal{P}_{\log}$ on a ball $B_{R/2} \equiv B_{R/2}(x_0) \subset \Omega$ and the corresponding energy with frozen coefficients, namely
\begin{equation}
    \int_{B_{R/2}} \left( a_i(R)|D\bu|^p + b_i(R) |D\bu|^p \log(e+|D\bu|) \right)\, dx,
\end{equation}
where
\begin{equation}\label{inf a, b}
    a_i(R)\equiv a_i(x_0, R) = \inf_{B_R} a(\cdot) \quad \text{ and } \quad b_i(R)\equiv b_i(x_0, R) = \inf_{B_R} b(\cdot).
\end{equation}
If $R \leq R_*$, then the concavity of $\omega_a$ and $\omega_b$ gives
$$
a_i(R) \geq a(x)-\omega_a(2R) \geq a(x)-2\omega_a(R) \quad \text{and} \quad b_i(R) \geq b(x)-2\omega_b(R)
$$
for every $x \in B_R$. Hence \eqref{cond : a,b} and \eqref{def : R_*} imply
\begin{equation}\label{ineq : ai+bi}
    a_i(R)+b_i(R) \geq \mu - 2\left(\omega_a(R)+\omega_b(R)\right) > \frac{\mu}{2} \quad \text{whenever } R \leq R_*.
\end{equation}
Under \eqref{cond : Logarithmic}, and hence \eqref{cond : Logarithmic 2}, the frozen energy turns out to be comparable to the full energy of a minimizer. Moreover, if $\ell = 0$ in \eqref{cond : ell}, the full energy can be viewed as a small perturbation of the frozen one. The proof of the next lemma uses the minimality of $\bu$ only through the reverse H\"older inequality \eqref{ineq : reverse holder}, together with the inequality \eqref{L log L} in $L \log L$ spaces.

\begin{lemma}\label{lem : Energy comparison}
    Let $\bu \in W^{1,p}(\Omega; \mr^N)$ be a local minimizer of $\mc{P}_{\log}$, let $B_R \equiv B_R(x_0) \subset \Omega$ be a ball with radius $R\leq R_*$, and let $b_i(R)$ be as in \eqref{inf a, b}. Moreover, let $\omega_b(\cdot)$ be as in \eqref{cond : modulus of b} and assume \eqref{cond : Logarithmic 2}. Then, for any $\gamma >0$, there exists a constant $c \equiv c(\texttt{data}, \gamma)$ such that
    \begin{equation}
    \begin{aligned}
        &\int_{B_{R/2}} [b(x)-b_i(R)]^{\gamma} |D\bu|^p \log^{\gamma}(e+|D\bu|)\, dx\\
        &\hspace{3cm}\leq c\left[ \omega_b(R)\log\left(\dfrac{1}{R}\right) \right]^{\gamma} \int_{B_R} H(x,D\bu)\, dx.
    \end{aligned}
    \end{equation}
\end{lemma}
\begin{proof}
    The proof is completely analogous to that of \cite[Lemma 5.1]{Baroni2015}, with the obvious modifications for vector-valued functions $\bu \in W^{1,p}(\Omega; \mr^N)$.
\end{proof}

\begin{remark}\label{rem : E-L}
Let $\bu \in W^{1,p}(\Omega;\mr^N)$ be a local minimizer of the functional $\mathcal{P}_{\log}$ defined in \eqref{main functional}. By standard arguments in Orlicz--Sobolev spaces as in \cite[Remark 5.1]{Baroni2015}, $\bu$ satisfies the weak form of the Euler--Lagrange equation
$$
\int_{\Omega} \langle a(x)\partial f(D\bu) + b(x)\partial g(D\bu), D\boldsymbol{\varphi} \rangle dx = 0
$$
for every test function $\boldsymbol{\varphi} \in W_0^{1,p}(\Omega;\mr^N)$ with $D\boldsymbol{\varphi} \in L^p \log L(\Omega;\mathbb{R}^{N\times n})$, where $f(\mathbf{z}) \coloneq |\mathbf{z}|^p$ and $g(\mathbf{z}) \coloneq |\mathbf{z}|^p \log(e+|\mathbf{z}|)$. This follows from the boundedness of $a(\cdot)$ and $b(\cdot)$ together with Young's inequality in Orlicz spaces for $g$ and its conjugate function $g^*$.
\end{remark}

We now compare $\bu$ with the minimizer of the functional with frozen coefficients.
\begin{lemma}\label{lem : comparison}
    Let $\bu\in W^{1,p}(\Omega;\mr^N)$ be a local minimizer of
    $\mathcal{P}_{\log}$ defined in \eqref{main functional}, and let $B_R \equiv B_R(x_0)\Subset\Omega$ be a ball with $R\le R_*$. Let $a_i(R)$ and $b_i(R)$ be as in \eqref{inf a, b}.
    Assume that $\omega_a$ and $\omega_b$ are moduli of continuity of $a$ and $b$, respectively, and that \eqref{cond : Logarithmic 2} holds.
    Let $\bv \in \bu + W_0^{1,p}(B_{R/2};\mr^N)$ be the minimizer of
    \begin{equation}\label{eq:frozen_problem}
        \mathbf{w} \mapsto \int_{B_{R/2}} \left( a_i(R)|D\mathbf{w}|^p + b_i(R)|D\mathbf{w}|^p\log(e+|D\mathbf{w}|) \right)\,dx.
    \end{equation}
    Then
    \begin{equation}\label{eq : comparison_result}
    \begin{aligned}
        &\int_{B_{R/2}} \left( a_i(R)|\bV_p(D\bu)-\bV_p(D\bv)|^2 + b_i(R)|\bV_{\log}(D\bu)-\bV_{\log}(D\bv)|^2 \right) \,dx \\
        &\hspace{3cm}\le c \left( \omega_a(R) + \omega_b(R) \log \left( \frac{1}{R} \right) \right) \int_{B_R} H(x, D\bu)\,dx,
    \end{aligned}
    \end{equation}
    where $c$ depends only on $\texttt{data}$.
\end{lemma}
\begin{proof}
    By the minimality of $\bv$, we have
    \begin{equation}
        \begin{aligned}
            &\int_{B_{R/2}} (a_i(R)|D\bv|^p + b_i(R)|D\bv|^p \log (e + |D\bv|)) \, dx \\
            &\hspace{3cm} \leq \int_{B_{R/2}} ( a_i(R) |D\bu|^p + b_i(R) |D\bu|^p \log (e + |D\bu|) ) \, dx.
        \end{aligned}
    \end{equation}
    We first derive an $L^p$ bound for $D\bv$. Set
    $$ \varphi(t) \coloneq a_i(R)t^p+b_i(R)t^p\log(e+t). $$
    Since $\log(e+t) \geq 1$, \eqref{ineq : ai+bi} gives $\varphi(t)\geq\frac{\mu}{2}t^p$. Moreover, $\varphi(|D\bu|) \leq H(x,D\bu)$ in $B_R$ by the definition of $a_i(R)$ and $b_i(R)$. Hence the minimality of $\bv$ yields
    \begin{equation}\label{ineq : Dv L^p}
    \begin{aligned}
        \int_{B_{R/2}}|D\bv|^p \, dx &\leq \frac{2}{\mu}\int_{B_{R/2}} \varphi(|D\bv|)\,dx \\
        &\leq \frac{2}{\mu}\int_{B_{R/2}} \varphi(|D\bu|)\,dx \leq \frac{2}{\mu}\int_{B_{R/2}}H(x,D\bu)\,dx.
    \end{aligned}
    \end{equation}

    By the monotonicity property of the vector fields in \eqref{eq : monotonicity V}, there exists a constant $c \ge 1$ depending only on $p$ such that
    $$
    \begin{aligned}
    &a_i(R)|\bV_p(D\bu) - \bV_p(D\bv)|^2 + b_i(R)|\bV_{\log}(D\bu) - \bV_{\log}(D\bv)|^2 \\
    &\le c \, \left\langle a_i(R)\partial f(D\bu) - a_i(R)\partial f(D\bv) + b_i(R)\partial g(D\bu) - b_i(R)\partial g(D\bv), \, D\bu - D\bv \right\rangle,
    \end{aligned}
    $$
    where $f$ and $g$ are as in Remark \ref{rem : E-L}. Since $\bv$ is the minimizer of the frozen functional \eqref{eq:frozen_problem}, it satisfies the Euler--Lagrange equation
    $$
    \int_{B_{R/2}} \left\langle a_i(R)\partial f(D\bv) + b_i(R)\partial g(D\bv), \, D\boldsymbol{\varphi} \right\rangle dx = 0 \quad \text{for all } \boldsymbol{\varphi} \in W_0^{1,p}(B_{R/2};\mathbb{R}^N).
    $$
    Integrating the previous inequality over $B_{R/2}$ and testing this equation with $\boldsymbol{\varphi} = \bu - \bv$, we obtain
    \begin{equation}\label{main model}
    \begin{aligned} 
        &\int_{B_{R/2}} \left( a_i(R)|\bV_p(D\bu) - \bV_p(D\bv)|^2 + b_i(R)|\bV_{\log}(D\bu) - \bV_{\log}(D\bv)|^2 \right) dx \\ 
        &\hspace{2cm}\le c \int_{B_{R/2}} \left\langle a_i(R)\partial f(D\bu) + b_i(R)\partial g(D\bu), \, D\bu - D\bv \right\rangle dx. 
    \end{aligned}
    \end{equation}
    Next, we test the Euler--Lagrange equation of $\bu$ in Remark \ref{rem : E-L} with the same function $\boldsymbol{\varphi} = \bu - \bv$:
    $$
    \int_{B_{R/2}} \left\langle a(x)\partial f(D\bu) + b(x)\partial g(D\bu), \, D\bu - D\bv \right\rangle dx = 0.
    $$
    Subtracting this identity from the right-hand side of \eqref{main model}, we can express it in terms of the oscillations $a(x) - a_i(R)$ and $b(x) - b_i(R)$:
    \begin{equation}\label{I_1 + I_2}
        \begin{aligned} 
        &\int_{B_{R/2}} \langle a_i(R)\partial f(D\bu) + b_i(R)\partial g(D\bu),\, D\bu - D\bv \rangle dx\\ 
        &\hspace{1.5cm}\le c \int_{B_{R/2}} (a(x) - a_i(R)) |\partial f(D\bu)| |D\bu - D\bv| \, dx \\ 
        &\hspace{1.5cm} + c \int_{B_{R/2}} (b(x) - b_i(R)) |\partial g(D\bu)| |D\bu - D\bv| \, dx \\ 
        &\hspace{1.5cm}\eqcolon I_1 + I_2. 
    \end{aligned}
    \end{equation}
    To estimate $I_1$, we first suppose that $a_i(R) > \omega_a(R)$. Using $0 \le a(x) - a_i(R) \le \omega_a(2R) \le 2\omega_a(R)$ in $B_R$, $|\partial f(D\bu)| = p|D\bu|^{p-1}$, and Young's inequality with $\varepsilon > 0$, we obtain
    $$
    \begin{aligned}
        I_1 &\le c \, \omega_a(R) \int_{B_{R/2}} |D\bu|^{p-1} |D\bu - D\bv|  \, dx \\
        &\le c\, \varepsilon \omega_a(R) \int_{B_{R/2}} |D\bu - D\bv|^p \,dx + c(\varepsilon) \, \omega_a(R) \int_{B_{R/2}}  |D\bu|^p \, dx\\
        &\le c\,\varepsilon \omega_a(R) \int_{B_{R/2}} |D \bu - D\bv|^p \,dx + c(\varepsilon)\, \omega_a(R) \int_{B_R} H(x, D\bu)\, dx.
    \end{aligned}
    $$
    We consider the first term on the right-hand side. If $p \geq 2$, then \eqref{2.10} gives
    $$
    \int_{B_{R/2}} |D \bu - D\bv|^p \,dx \leq c \int_{B_{R/2}} |\bV_p(D\bu) - \bV_p(D\bv)|^2 \, dx.
    $$
    If $1<p<2$, then Young's inequality, \eqref{ineq : |V_p|}, and \eqref{ineq : Dv L^p} give
    $$
    \begin{aligned}
        &\int_{B_{R/2}} |D \bu - D\bv|^p \,dx \\
        &= \int_{B_{R/2}} (|D\bu| + |D\bv|)^{\frac{p(p-2)}{2}} |D \bu - D\bv|^p (|D\bu| + |D\bv|)^{\frac{p(2-p)}{2}} \,dx\\
        &\leq c \int_{B_{R/2}} (|D\bu| + |D\bv|)^{p-2} |D \bu - D\bv|^2 \, dx + c \int_{B_{R/2}}(|D\bu| + |D\bv|)^{p} \, dx\\
        &\leq c \int_{B_{R/2}} |\bV_p(D\bu) - \bV_p(D\bv)|^2 \, dx + c\int_{B_{R}} H(x,D\bu) \, dx.
    \end{aligned}
    $$
    Combining the above estimates and using $\omega_a(R) < a_i(R)$, we obtain in both cases
    $$
    I_1 \leq c\,\varepsilon a_i(R) \int_{B_{R/2}} |\bV_p(D\bu)-\bV_p(D\bv)|^2\,dx + C \omega_a(R) \int_{B_R}H(x,D\bu)\,dx,
    $$
    where $c$ depends only on $n$, $N$, and $p$, and $C$ depends only on $\varepsilon$ and $\texttt{data}$.
    Suppose now that $a_i(R) \le \omega_a(R)$. Then Young's inequality, the triangle inequality, and \eqref{ineq : Dv L^p} give
    $$
    \begin{aligned}
        I_1 &\leq c \, \omega_a(R) \int_{B_{R/2}} |D\bu|^{p-1} |D\bu - D\bv|  \, dx \\
        &\leq c \, \omega_a(R) \int_{B_{R/2}} |D\bu|^{p-1} (|D\bu| + |D\bv|)  \, dx \\
        &= c \, \omega_a(R) \int_{B_{R/2}} |D\bu|^p + |D\bu|^{p-1}\,|D\bv|  \, dx \\
        &\leq c \, \omega_a(R) \int_{B_{R/2}} |D\bu|^p + (|D\bu|^p + |D\bv|^p)  \, dx \\
        &\leq c \, \omega_a(R) \int_{B_{R/2}} (|D\bu|^p + |D\bv|^p)  \, dx \\
        &\leq c \, \omega_a(R) \int_{B_{R}} H(x,D\bu)  \, dx.
    \end{aligned}
    $$
    In both cases, we have
    $$
        I_1 \leq c\,\varepsilon a_i(R) \int_{B_{R/2}} |\bV_p(D\bu)-\bV_p(D\bv)|^2\,dx + C \omega_a(R) \int_{B_R}H(x,D\bu)\,dx.
    $$
    To estimate $I_2$, we recall that $g(s) = s^p \log(e+s)$ and use $0 \le b(x)-b_i(R) \le 2\omega_b(R)$ in $B_R$ to obtain
    $$
    \begin{aligned} 
        I_2 &\le c \, \omega_b(R) \int_{B_{R/2}} |D\bu|^{p-1}\log(e+|D\bu|) \, |D\bu - D\bv| \, dx.
    \end{aligned}
    $$
    By Young's inequality with $\delta>0$, the monotonicity of the logarithm, and \eqref{ineq : 2.5}, we have
    $$
    \begin{aligned}
        &|D\bu|^{p-1}\log(e+|D\bu|) \, |D\bu - D\bv| \\
        &= \left( (|D\bu| + |D\bv|)^{\frac{p-2}{2}} \log^{\frac{1}{2}}(e + |D\bu| + |D\bv| ) |D\bu - D\bv| \right) \\
        &\hspace{2cm} \times\left( \frac{(|D\bu| + |D\bv|)^{\frac{2-p}{2}} |D\bu|^{p-1} \log(e+|D\bu|)} {\log^{\frac{1}{2}}(e + |D\bu| + |D\bv| )} \right)\\
        &\leq \frac{\delta}{2} \left( (|D\bu| + |D\bv|)^{p-2} \log (e + |D\bu| + |D\bv| ) |D\bu - D\bv|^2 \right)\\
        &\hspace{2cm} + \frac{1}{2\delta}\left( \frac{(|D\bu| + |D\bv|)^{2-p} |D\bu|^{2(p-1)} \log^2(e+|D\bu|)} {\log(e + |D\bu| + |D\bv| )} \right)\\
        &\leq \frac{c\, \delta}{2} |\bV_{\log}(D\bu) - \bV_{\log}(D\bv)|^2\\
        &\hspace{2cm} + \frac{c}{2\delta} (|D\bu| + |D\bv|)^{2-p} |D\bu|^{2(p-1)} \log (e+|D\bu|)
    \end{aligned}
    $$
    For the last term, if $1<p<2$, then Young's inequality gives
    $$
    \begin{aligned}
        &(|D\bu| + |D\bv|)^{2-p} |D\bu|^{2(p-1)} \log (e+|D\bu|)\\ 
        &\hspace{2cm}\leq c(|D\bu|^{2-p} + |D\bv|^{2-p}) |D\bu|^{2(p-1)} \log (e+|D\bu|)\\
        &\hspace{2cm}\leq c (|D\bu|^{2-p}|D\bu|^{2(p-1)} + |D\bv|^{2-p}|D\bu|^{2(p-1)}) \log (e+|D\bu|)\\
        &\hspace{2cm}\leq c (|D\bu|^{p} + |D\bv|^p + |D\bu|^p) \log (e+|D\bu|)\\
        &\hspace{2cm}\leq c (|D\bu|^{p} + |D\bv|^p) \log (e+|D\bu|).
    \end{aligned} 
    $$
    If $p \ge 2$, then
    $$
    \begin{aligned}
        (|D\bu| + |D\bv|)^{2-p} |D\bu|^{2(p-1)} \log (e+|D\bu|) &\le |D\bu|^{2-p} |D\bu|^{2(p-1)} \log(e + |D\bu|)\\
        &\le |D\bu|^{p} \log(e + |D\bu|).
    \end{aligned}
    $$
    In both cases, we obtain
    $$
    \begin{aligned}
        &(|D\bu| + |D\bv|)^{2-p} |D\bu|^{2(p-1)} \log (e+|D\bu|)\\ 
        &\hspace{2.5cm}\le c(|D\bu|^{p} + |D\bv|^p) \log (e+|D\bu|)\\
        &\hspace{2.5cm}\leq c\left( |D\bu|^{p} \log (e+|D\bu|) + |D\bv|^p \log (e+|D\bu|) \right)\\
        &\hspace{2.5cm}\le c\left( \, g(|D\bu|) + g(|D\bv|) \right).
    \end{aligned}
    $$
    Hence,
    $$
    \begin{aligned}
    I_2 &\leq c\,\omega_b(R) \bigg[\delta \int_{B_{R/2}}|\bV_{\log}(D\bu) - \bV_{\log}(D\bv)|^2 \,dx\\
    &\qquad + \dfrac{c}{\delta} \int_{B_{R/2}} g(|D\bu|) + g(|D\bv|) \, dx \bigg].
    \end{aligned}
    $$
    If $b_i(R) \ge \omega_b(R)$, we choose
    $$
    \delta\coloneq\frac{\varepsilon b_i(R)}{c\,\omega_b(R)}.
    $$
    Then
    $$
    c\,\delta\,\omega_b(R)=\varepsilon\, b_i(R),
    $$
    and since $\frac{\omega_b(R)}{b_i(R)} \le 1$, we have
    $$
    c\,\frac{\omega_b(R)^2}{\varepsilon b_i(R)} \leq \dfrac{c}{\varepsilon}\omega_b(R).
    $$
    Therefore,
    $$
    \begin{aligned}
    I_2
    &\le c \,\varepsilon  b_i(R)\int_{B_{R/2}}|\bV_{\log}(D\bu)-\bV_{\log}(D\bv)|^2\,dx\\
    &\hspace{1cm}+c(\varepsilon)\omega_b(R)\int_{B_{R/2}} g(|D\bu|) + g(|D\bv|) \, dx .
    \end{aligned}
    $$
    If $b_i(R) < \omega_b(R)$, we choose $\delta = 1$ and obtain
    $$
    I_2 \leq c\,\omega_b(R)\left[ \int_{B_{R/2}}|\bV_{\log}(D\bu) - \bV_{\log}(D\bv)|^2 \,dx + \int_{B_{R/2}} g(|D\bu|) + g(|D\bv|) \, dx  \right]. 
    $$ 
    By \eqref{ineq : comparable},
    $$
    |\bV_{\log}(D\bu)-\bV_{\log}(D\bv)|^2 \le c\left(|\bV_{\log}(D\bu)|^2+|\bV_{\log}(D\bv)|^2\right) \le c\left(g(|D\bu|) + g(|D\bv|) \right).
    $$
    Hence, 
    $$
    I_2 \leq c\,\omega_b(R)\left[ \int_{B_{R/2}} g(|D\bu|) + g(|D\bv|) \, dx\right] .
    $$
    In both cases, we get
    $$
    \begin{aligned}
    I_2 &\leq c\,\varepsilon b_i(R) \int_{B_{R/2}}|\bV_{\log}(D\bu) - \bV_{\log}(D\bv)|^2 \,dx\\
    &\quad + c\, \omega_b(R) \int_{B_{R/2}} g(|D\bu|) + g(|D\bv|)  \, dx.
    \end{aligned}
    $$
    For the term involving $D\bu$, applying \eqref{L log L} with $\mathbf{f}=|D\bu|^p$, $\gamma=1$ and $q=1+\delta_g$, we obtain
    \begin{equation}\label{ineq : LlogL Du}
    \dashint_{B_{R/2}}|D\bu|^p\log\left(e+\frac{|D\bu|^p}{(|D\bu|^p)_{B_{R/2}}}\right)\,dx\le c \left(\dashint_{B_{R/2}}|D\bu|^{p(1+\delta_g)}\,dx\right)^{\frac{1}{1+\delta_g}}.        
    \end{equation}
    By the higher integrability estimate \eqref{ineq : reverse holder} and $\mu|D\bu|^p \le H(x,D\bu)$, we have
    \begin{equation}\label{ineq : higher int Du}
    \left(\dashint_{B_{R/2}}|D\bu|^{p(1+\delta_g)}\,dx\right)^{\frac{1}{1+\delta_g}}\le c \dashint_{B_R}H(x,D\bu)\,dx.        
    \end{equation}
    We set $A \coloneq (|D\bu|^p)_{B_{R/2}} \leq c\,R^{-n} \|D\bu\|^p_{L^p}$. By \eqref{cond : log}, we have
    $$
    \begin{aligned}
        \log(e + |D\bu|) &= \log\left(e + A^{\frac{1}{p}} \left( \frac{|D\bu|^p}{A} \right)^{\frac{1}{p}}\right)\\
        &\leq \log\left(e + A^{\frac{1}{p}}\right)+ \log\left(e +\left( \frac{|D\bu|^p}{A}\right)^{\frac{1}{p}} \right)\\
        &\leq c \left[ \log(e + A) + \log\left(e + \frac{|D\bu|^p}{A}\right) \right] 
    \end{aligned}
    $$
    and hence
    $$
    |D\bu|^p \log(e + |D\bu|) \le c \left( |D\bu|^p \log\left(e + \frac{|D\bu|^p}{A}\right)  + |D\bu|^p \log(e + A) \right).
    $$
    Since $R \leq \frac{1}{e}$, we have $R^{-n} \geq 1$ and $\log\left(\frac{1}{R}\right) \geq 1$, and hence
    $$
    \begin{aligned}
        \log(e + A) &\leq \log\left(R^{-n}\left(e+ c\,\|D\bu\|^p_{L^p}\right)\right)\\
        &= \log \left( e + c\,\|D\bu\|^p_{L^p} \right) + n \log\left(\frac{1}{R}\right)\\
        &\leq \left(\log\left(e + c\,\|D\bu\|^p_{L^p}\right) + n\right) \log\left( \frac{1}{R} \right)\\
        &=c\log\left( \frac{1}{R} \right),
    \end{aligned}
    $$
    where $c$ depends only on $\texttt{data}$.
    Hence, by \eqref{ineq : LlogL Du} and \eqref{ineq : higher int Du},
    $$
    \begin{aligned}
        &\int_{B_{R/2}} g(|D\bu|) \, dx = \int_{B_{R/2}} |D\bu|^p \log(e + |D\bu|) \,dx\\
        &\hspace{1cm} \le c \left( \int_{B_{R/2}} |D\bu|^p \log\left(e + \frac{|D\bu|^p}{A}\right)\,dx  + \int_{B_{R/2}}|D\bu|^p \log(e + A)\,dx \right)\\
        &\hspace{1cm}\le c \left( \int_{B_{R}} H(x,D\bu)\, dx + \log(e + A) \int_{B_{R/2}} |D\bu|^p \, dx \right)\\
        &\hspace{1cm}\le c \left( \int_{B_{R}} H(x,D\bu)\, dx + \log\left( \frac{1}{R} \right) \int_{B_{R}} H(x,D\bu) \, dx \right)\\
        &\hspace{1cm}\le c \log\left( \frac{1}{R} \right) \int_{B_{R}} H(x,D\bu)\, dx.
    \end{aligned}
    $$
    For the term involving $D\bv$, let $\sigma_0$ be as in Theorem \ref{thm:hi-phi} and fix $\sigma \coloneq \min\{\sigma_0,\delta_g\}$. Since $\varphi(|D\bu|) \leq H(x,D\bu)$ in $B_R$, Theorem \ref{thm:hi-phi} applied on $B=B_{R/2}$ with $a_0=a_i(R)$ and $b_0=b_i(R)$, together with \eqref{ineq : reverse holder}, yields
    $$
    \begin{aligned}
    \left(\dashint_{B_{R/2}}|D\bv|^{p(1+\sigma)}dx\right)^{\frac{1}{1+\sigma}}
    &\le\frac{2}{\mu}\left(\dashint_{B_{R/2}}\varphi(D\bv)^{1+\sigma}dx\right)^{\frac{1}{1+\sigma}}\\
    &\le c\left(\dashint_{B_{R/2}}H(x,D\bu)^{1+\sigma}dx\right)^{\frac{1}{1+\sigma}}\\
    &\le c\,\dashint_{B_R}H(x,D\bu)\,dx.
    \end{aligned}
    $$
    We now set $A_{\bv} \coloneq (|D\bv|^p)_{B_{R/2}}$. By \eqref{ineq : Dv L^p}, we have $A_{\bv} \leq c\,R^{-n}\|H(\cdot,D\bu)\|_{L^1(\Omega)}$, and the argument used above for $\log(e+A)$ gives $\log(e+A_{\bv}) \leq c\log\left(\frac{1}{R}\right)$ with $c$ depending only on $\texttt{data}$. Applying \eqref{L log L} with $\mathbf{f}=|D\bv|^p$, $\gamma = 1$, and $q = 1 + \sigma$, and arguing exactly as for $D\bu$ with $A$ replaced by $A_{\bv}$, we obtain
    $$
    \int_{B_{R/2}}g(|D\bv|)\,dx\le c\log\left(\frac{1}{R}\right) \int_{B_R}H(x,D\bu)\,dx.
    $$
    Consequently,
    $$
    \begin{aligned}
        \int_{B_{R/2}} \left( g(|D\bu|) + g(|D\bv|) \right)\,dx 
        \le c \log\left(\frac{1}{R}\right) \int_{B_R} H(x,D\bu)\,dx.    
    \end{aligned}
    $$
    Therefore,
    $$
    I_2 \leq c\,\varepsilon b_i(R)\int_{B_{R/2}}|\bV_{\log}(D\bu) - \bV_{\log}(D\bv)|^2 \,dx + c \,\omega_b(R) \log\left( \frac{1}{R} \right) \int_{B_R} H(x,D\bu)\,dx. 
    $$
    Combining \eqref{main model} and \eqref{I_1 + I_2}, we obtain
    $$
    \begin{aligned}
        &\int_{B_{R/2}} \left( a_i(R)|\bV_p(D\bu) - \bV_p(D\bv)|^2 + b_i(R)|\bV_{\log}(D\bu) - \bV_{\log}(D\bv)|^2 \right) dx \\ 
        &\le c\, \varepsilon a_i(R) \int_{B_{R/2}} |\bV_p(D\bu)-\bV_p(D\bv)|^2\,dx + c \,\omega_a(R) \int_{B_R}H(x,D\bu)\,dx\\
        &+  c\,\varepsilon b_i(R)\int_{B_{R/2}}|\bV_{\log}(D\bu) - \bV_{\log}(D\bv)|^2 \,dx + c \,\omega_b(R) \log\left( \frac{1}{R} \right) \int_{B_R} H(x,D\bu)\,dx.
    \end{aligned}
    $$
    Choosing $\varepsilon>0$ so small that $c\,\varepsilon \leq \frac{1}{2}$, we can absorb the terms involving $\varepsilon$ into the left-hand side. Hence,
    $$
    \begin{aligned}
        &\int_{B_{R/2}} \left( a_i(R)|\bV_p(D\bu) - \bV_p(D\bv)|^2 + b_i(R)|\bV_{\log}(D\bu) - \bV_{\log}(D\bv)|^2 \right) dx \\ 
        &\hspace{3.3cm}\le c \left[\omega_a(R) + \omega_b(R) \log\left( \frac{1}{R} \right) \right] \int_{B_R}H(x,D\bu)\,dx.
    \end{aligned}
    $$
\end{proof}

Using Lemma \ref{lem : comparison}, we next derive a decay estimate for local minimizers of $\mathcal{P}_{\log}$.
\begin{lemma}\label{lem : decay}
    Let $\bu \in W^{1,p}(\Omega;\mr^N)$ be a local minimizer of the functional $\mathcal{P}_{\log}$ defined in \eqref{main functional}, and let $B_R \equiv B_R(x_0)\Subset\Omega$ be a ball with $R \le R_*$. Let $a_i(R)$ and $b_i(R)$ be as in \eqref{inf a, b}.
    Assume that $\omega_a$ and $\omega_b$ are moduli of continuity of $a$ and $b$, respectively, and that \eqref{cond : Logarithmic 2} holds. Then
    \begin{equation}\label{eq : decay}
        \int_{B_{\rho}} H(x,D\bu) \, dx \le c_d \left[ \left( \frac{\rho}{R} \right)^n + \omega_a(R) + \omega_b(R) \log \left( \frac{1}{R} \right) \right] \int_{B_R} H(x, D\bu)\, dx
    \end{equation}
    holds for every $0 < \rho \leq R$, where $B_{\rho} \equiv B_{\rho}(x_0)$ and $c_d$ depends only on $\texttt{data}$.
\end{lemma}
\begin{proof} 
    It suffices to prove \eqref{eq : decay} for $\rho \leq \frac{R}{4}$, since otherwise it holds with $c_d = 4^n$. By the definition of $H(x, D\bu)$, we have
    \begin{equation*}
    \begin{aligned}
        &\int_{B_{\rho}} H(x, D\bu) \, dx = \int_{B_{\rho}} a(x) |D\bu|^p + b(x)|D\bu|^p \log(e + |D\bu|) \, dx \\
        &= \int_{B_{\rho}} [a(x) - a_i(R)] |D\bu|^p + a_i(R) |D\bu|^p \\
        &\hspace{1.5cm} + b_i(R) |D\bu|^p \log(e + |D\bu|) + [b(x)-b_i(R)]|D\bu|^p \log(e + |D\bu|) \, dx\\
        &\le\int_{B_{\rho}} a_i(R) |D\bu|^p + b_i(R) |D\bu|^p \log(e + |D\bu|) \, dx\\
        &\hspace{1.5cm} + \int_{B_{R/2}} [a(x) - a_i(R)] |D\bu|^p + [b(x)-b_i(R)]|D\bu|^p \log(e + |D\bu|) \, dx.
    \end{aligned}
    \end{equation*}
    By Lemma \ref{lem : Energy comparison} with $\gamma = 1$ and $0 \le a(x) - a_i(R) \le 2\omega_a(R)$ in $B_R$, the second integral on the right-hand side satisfies 
    \begin{equation*}
    \begin{aligned}
        &\int_{B_{R/2}} [a(x) - a_i(R)] |D\bu|^p + [b(x)-b_i(R)]|D\bu|^p \log(e + |D\bu|) \, dx\\
        &\leq c\, \omega_a(R) \int_{B_R} H(x, D\bu)\,dx + c\,\omega_b(R)\log\left( \frac{1}{R} \right) \int_{B_R} H(x,D\bu) \, dx\\
        &=c \left( \omega_a(R) + \omega_b(R) \log\left( \frac{1}{R} \right) \right) \int_{B_R} H(x, D\bu) \, dx,
    \end{aligned}
    \end{equation*}
    where $c$ depends only on $\texttt{data}$.
    Next, by the equivalence between the frozen energy and the vector fields $\bV_p$ and $\bV_{\log}$ and by the elementary inequality $|\mathbf{P}|^2 \le 2|\mathbf{Q}|^2 + 2|\mathbf{P}-\mathbf{Q}|^2$, we obtain
    \begin{equation*}
    \begin{aligned}
        &\int_{B_{\rho}} a_i(R) |D\bu|^p + b_i(R) |D\bu|^p \log(e + |D\bu|) \, dx\\
        &\leq c \int_{B_{\rho}} a_i(R) |\bV_p(D\bu)|^2 + b_i(R) |\bV_{\log} (D\bu)|^2 \, dx\\
        &\leq c \int _{B_{\rho}} a_i(R) |\bV_p(D\bv)|^2 + b_i(R) |\bV_{\log} (D\bv)|^2 \, dx\\
        &\hspace{0.5cm} + c \int_{B_{R/2}} a_i(R) |\bV_p(D\bu) - \bV_p(D\bv)|^2 + b_i(R) |\bV_{\log} (D\bu) - \bV_{\log} (D\bv)|^2 \, dx.
    \end{aligned}
    \end{equation*}
    By Lemma \ref{lem : comparison}, the last integral satisfies
    \begin{equation*}
    \begin{aligned}
        &\int_{B_{R/2}} a_i(R) |\bV_p(D\bu) - \bV_p(D\bv)|^2 + b_i(R) |\bV_{\log} (D\bu) - \bV_{\log} (D\bv)|^2 \, dx\\
        &\leq c\left( \omega_a(R) + \omega_b(R) \log\left( \frac{1}{R} \right) \right) \int_{B_R} H(x, D\bu) \, dx.
    \end{aligned}
    \end{equation*}
    For the first integral on the right-hand side, since $\rho \leq \frac{R}{4}$, the estimate \eqref{ineq : sup, int} for $\bv$ in $B_{R/2}$, the minimality of $\bv$, and \eqref{inf a, b} imply that
    \begin{align*}
    &\dashint_{B_\rho} \left( a_i(R)|\bV_p(D\bv)|^2 + b_i(R)|\bV_{\log}(D\bv)|^2 \right) dx\\ 
    &\leq c \,\dashint_{B_{R/2}} \left( a_i(R)|\bV_p(D\bu)|^2 + b_i(R)|\bV_{\log}(D\bu)|^2 \right) dx\\
    &\leq c \,\dashint_{B_R} H(x, D\bu) \, dx = \frac{c}{|B_R|} \int_{B_R} H(x, D\bu) \, dx.
    \end{align*}
    Multiplying both sides by $|B_\rho|$ and using $\frac{|B_\rho|}{|B_R|} = \left(\frac{\rho}{R}\right)^n$, we get
    \begin{align*}
        \int_{B_\rho} \left( a_i(R)|\bV_p(D\bv)|^2 + b_i(R)|\bV_{\log}(D\bv)|^2 \right) dx
        \le c \left( \frac{\rho}{R} \right)^n \int_{B_R} H(x, D\bu) \, dx.
    \end{align*}
    Combining all the above estimates, we obtain \eqref{eq : decay}.
\end{proof}

We now prove assertion (i) of Theorem \ref{thm : main theorem}. Let $\bu \in W^{1,p}(\Omega;\mr^N)$ be a local minimizer of $\mathcal{P}_{\log}$ and assume that $\ell = 0$ in \eqref{cond : ell}. Together with \eqref{cond : omega_a}, this means that
\begin{equation}\label{cond : omega zero}
\limsup_{r \to 0} \left( \omega_a(r) + \omega_b(r) \log \left( \frac{1}{r} \right) \right) = 0.
\end{equation}

Our goal is to show that $\bu \in C^{0,\beta}_{\mathrm{loc}}(\Omega;\mathbb{R}^N)$ for every $\beta \in (0,1)$. More precisely, we prove the following local estimate: for every $\beta \in (0,1)$, there exists $R_0 \in (0, R_*]$, depending only on $\texttt{data}$, $\omega_a$, $\omega_b$, and $\beta$, such that
\begin{equation}\label{eq:5.18}
[\bu]_{\beta;B_{R/2}} \le c\left(\texttt{data}, \beta\right) \left( R^{\,p\,(1-\beta)} \dashint_{B_R} H(x, D\bu) \, dx \right)^{\frac{1}{p}}
\end{equation}
for every ball $B_R \Subset \Omega$ with $R \le R_0$.
Here, $[\bu]_{\beta;B_{R/2}}$ denotes the usual H\"older seminorm of $\bu$ in the ball $B_{R/2}$:
\begin{equation}\label{eq:5.19}
[\bu]_{\beta;B_{R/2}} \coloneq \sup_{\substack{x, y \in B_{R/2}\\ x \neq y}} \frac{|\bu(x) - \bu(y)|}{|x - y|^\beta}.
\end{equation}

To establish this, we first prove a Morrey-type estimate. Specifically, for every $\delta \in (0, n)$, there exist a positive constant $c = c(\texttt{data}, \delta)$ and a radius
\begin{equation}\label{R_o choice}
R_0 \in (0, R_*], \quad \text{depending only on } \texttt{data}, \ \omega_a, \ \omega_b, \text{ and } \delta,
\end{equation}
such that the decay estimate
\begin{equation}\label{ineq : decay estimate}
\int_{B_\rho} H(x, D\bu) \, dx \le c \left( \frac{\rho}{R} \right)^{n-\delta} \int_{B_R} H(x, D\bu) \, dx
\end{equation}
holds for all concentric balls $B_\rho \subset B_R \Subset \Omega$ with $0 < \rho \le R \le R_0$.

The Poincar\'e inequality \eqref{vector poincare} and \eqref{ineq : decay estimate} give
$$
\int_{B_\rho} \left| \frac{\bu - (\bu)_{B_\rho}}{\rho} \right|^p dx \le c \int_{B_\rho} |D\bu|^p \, dx \le c \left( \frac{\rho}{R} \right)^{n-\delta} \int_{B_R} H(x, D\bu) \, dx
$$
for all $0 < \rho \le R \le R_0$. Hence the $C_{\mathrm{loc}}^{0,\beta}$-regularity of $\bu$ with $\beta = 1 - \frac{\delta}{p}$ follows from Campanato's characterization of H\"older continuity and a standard covering argument. The same argument also yields the local estimate \eqref{eq:5.18} for $\beta = 1 - \frac{\delta}{p}$, and hence for every smaller exponent, because $[\bu]_{\beta';B_{R/2}} \le R^{\beta-\beta'}[\bu]_{\beta;B_{R/2}}$ whenever $0<\beta'\le\beta$.
Since $\delta > 0$ can be chosen arbitrarily small, we conclude that $\bu \in C^{0,\beta}_{\mathrm{loc}}(\Omega;\mr^N)$ for every $\beta \in (0,1)$.

It remains to prove \eqref{ineq : decay estimate}. We apply Lemma \ref{lemma 2.5} to
$$
\phi(\rho) \coloneq \int_{B_\rho} H(x, D\bu) \, dx,
$$
with $\tilde{c} \coloneq c_d$ from Lemma \ref{lem : decay}. For a given $\delta \in (0,n)$, let $\overline{\varepsilon} > 0$ be the constant from Lemma \ref{lemma 2.5}, which depends only on $n$, $\delta$, and $c_d$, and hence only on $\texttt{data}$ and $\delta$. By \eqref{cond : omega zero}, we can choose $R_0 \in (0, R_*]$ such that
$$
\omega_a(R) + \omega_b(R) \log\left(\frac{1}{R}\right) \le \overline{\varepsilon} \quad \text{for all } R \le R_0.
$$
Then Lemma \ref{lem : decay} gives
$$
\phi(\rho) \le c_d \left[ \left(\frac{\rho}{R}\right)^n + \overline{\varepsilon} \right] \phi(R)
$$
whenever $0 < \rho \le R \le R_0$, and \eqref{ineq : decay estimate} follows from Lemma \ref{lemma 2.5}.

\medskip

We now prove assertion (ii) of Theorem \ref{thm : main theorem}.
Assume that $a(\cdot)$ and $b(\cdot)$ are H\"older continuous with exponent $\alpha \in (0,1)$, that is,
$$
\omega_a(R) \le C R^\alpha \qquad \text{and} \qquad \omega_b(R) \le C R^\alpha.
$$
In particular, \eqref{cond : omega zero} holds, so the Morrey-type estimate \eqref{ineq : decay estimate} is available. We derive a decay estimate for the excess $\dashint_{B_\rho} |D\bu - (D\bu)_{B_\rho}|^p dx$ and conclude that $D\bu \in C^{0,\beta}_{\text{loc}}(\Omega;\mathbb{R}^{N\times n})$ by Campanato's characterization.

Let $B_R \equiv B_R(x_0) \Subset \Omega$ with $R \le R_*$, and let $\bv \in \bu + W_0^{1,p}(B_{R/2}; \mathbb{R}^N)$ be the minimizer of the frozen functional \eqref{eq:frozen_problem}, with $a_i(R)$ and $b_i(R)$ as in \eqref{inf a, b}.

For $0 < \rho \le \frac{R}{2}$, the triangle inequality and the elementary inequality
$$
\dashint_{B_\rho} |D\bu - (D\bu)_{B_\rho}|^p dx \le 2^p \dashint_{B_\rho} |D\bu - \mathbf{z}|^p dx,
$$
valid for every $\mathbf{z} \in \mr^{N \times n}$, give
$$
\begin{aligned}
    \dashint_{B_\rho} |D\bu - (D\bu)_{B_\rho}|^p dx &\le 2^p \dashint_{B_\rho} |D\bu - (D\bv)_{B_\rho}|^p dx\\
    &\leq 4^p {\dashint_{B_\rho} |D\bv - (D\bv)_{B_\rho}|^p dx} + 4^p {\dashint_{B_\rho} |D\bu - D\bv|^p dx}.
\end{aligned}
$$

Theorem \ref{Thm 3.1} applied on $B_{R/2}$ with $a_0=a_i(R)$ and $b_0=b_i(R)$, the minimality of $\bv$, and \eqref{inf a, b} give
$$
\begin{aligned}
    &(a_i(R)+b_i(R))\dashint_{B_\rho} |D\bv - (D\bv)_{B_\rho}|^p dx\\
    &\hspace{1cm}\le c \left(\frac{\rho}{R}\right)^{\tilde{\alpha}p} \dashint_{B_{R/2}} \left( a_i(R)|D\bv|^p + b_i(R)|D\bv|^p \log(e+|D\bv|) \right) dx\\
    &\hspace{1cm}\le c \left(\frac{\rho}{R}\right)^{\tilde{\alpha}p} \dashint_{B_{R/2}} \left( a_i(R)|D\bu|^p + b_i(R)|D\bu|^p \log(e+|D\bu|) \right) dx\\
    &\hspace{1cm}\le c \left(\frac{\rho}{R}\right)^{\tilde{\alpha}p} \dashint_{B_R} H(x, D\bu) dx
\end{aligned}
$$
for some $\tilde{\alpha} \in (0,1)$ depending only on $n$, $N$, and $p$. By \eqref{ineq : ai+bi}, this implies
$$
\dashint_{B_\rho} |D\bv - (D\bv)_{B_\rho}|^p dx \le c \left(\frac{\rho}{R}\right)^{\tilde{\alpha}p} \dashint_{B_R} H(x, D\bu) dx.
$$

To estimate the comparison error, we note that the function $t \mapsto \frac{1}{p}t^p\log(e+t)$, whose vector field is $\bV_{\log}$, satisfies \eqref{cond : mono map} with a constant depending only on $p$, and that its second derivative is bounded from below by $(p-1)t^{p-2}$ because $\log(e+t) \geq 1$. Hence \eqref{ineq : |V_p|} and \eqref{ineq : 2.5} give
$$
|\bV_p(\mathbf{Q})-\bV_p(\mathbf{R})|^2 \leq c\,(|\mathbf{Q}|+|\mathbf{R}|)^{p-2}|\mathbf{Q}-\mathbf{R}|^2 \leq c\,|\bV_{\log}(\mathbf{Q})-\bV_{\log}(\mathbf{R})|^2
$$
for all $\mathbf{Q},\mathbf{R}\in\mr^{N\times n}$, where $c$ depends only on $n$, $N$, and $p$. Combining this with \eqref{ineq : ai+bi} and Lemma \ref{lem : comparison}, we obtain
\begin{equation}\label{ineq : V_p comparison}
\begin{aligned}
&\dashint_{B_{R/2}} |\bV_p(D\bu)-\bV_p(D\bv)|^2 \,dx \\
&\leq \frac{c}{\mu}\dashint_{B_{R/2}} \left( a_i(R)|\bV_p(D\bu)-\bV_p(D\bv)|^2 + b_i(R)|\bV_{\log}(D\bu)-\bV_{\log}(D\bv)|^2 \right) dx\\
&\leq c\left( \omega_a(R) + \omega_b(R) \log\left(\frac{1}{R}\right) \right) \dashint_{B_R} H(x, D\bu)\,dx.
\end{aligned}
\end{equation}
If $p \ge 2$, then \eqref{2.10} and \eqref{ineq : V_p comparison} give
$$
\begin{aligned}
    \dashint_{B_{R/2}} |D\bu - D\bv|^p dx &\le c \dashint_{B_{R/2}} |\bV_p(D\bu) - \bV_p(D\bv)|^2 dx \\
    &\le c \, \left( \omega_a(R) + \omega_b(R) \log\left(\frac{1}{R}\right) \right) \dashint_{B_R} H(x, D\bu) dx \\
    &\le c R^{\sigma_1} \dashint_{B_R} H(x, D\bu) dx \eqcolon I
\end{aligned}
$$
for some $\sigma_1 \in (0, \alpha)$.

If $1 < p < 2$, then H\"older's inequality, \eqref{ineq : V_p comparison}, and \eqref{ineq : Dv L^p} give
$$
\begin{aligned}
    &\dashint_{B_{R/2}} |D\bu - D\bv|^p dx \\
    &\hspace{1cm} \le c \left( \dashint_{B_{R/2}} |\bV_p(D\bu) - \bV_p(D\bv)|^2 dx \right)^{\frac{p}{2}} \left( \dashint_{B_{R/2}} (|D\bu| + |D\bv|)^p dx \right)^{\frac{2-p}{2}}  \\
    &\hspace{1cm}\le c \left[ \left( \omega_a(R) + \omega_b(R) \log\left(\frac{1}{R}\right) \right) \dashint_{B_R} H(x, D\bu) dx \right]^{\frac{p}{2}} \left( \dashint_{B_R} H(x, D\bu) dx \right)^{\frac{2-p}{2}}\\
    &\hspace{1cm} \le c R^{\sigma_2} \dashint_{B_R} H(x, D\bu) dx \eqcolon II
\end{aligned}
$$
where $\sigma_2 = \frac{p}{2} \cdot \sigma_1 > 0$.

Combining these estimates, we obtain the oscillation estimate
$$
\dashint_{B_\rho} |D\bu - (D\bu)_{B_\rho}|^p dx \le c \left[ \left(\frac{\rho}{R}\right)^{\tilde{\alpha}p} + R^\sigma \left(\frac{R}{\rho}\right)^n \right] \dashint_{B_R} H(x, D\bu) dx
$$
for $0<\rho\le\frac{R}{2}$, where $\sigma = \min\{\sigma_1, \sigma_2\} > 0$. The same estimate holds for $\frac{R}{2}<\rho\le R$, since in this case its left-hand side is bounded by $\frac{2^{p+n}}{\mu}\dashint_{B_R} H(x, D\bu) dx$ and $\frac{\rho}{R}>\frac{1}{2}$.

Let $\delta \in (0, \sigma)$ be chosen below, and let $R_0 \in (0,R_*]$ be the radius in \eqref{R_o choice} corresponding to $\delta$. For a fixed ball $B_{R_0} \Subset \Omega$ and any $0 < R \leq R_0$, the Morrey-type estimate \eqref{ineq : decay estimate} gives
$$
\int_{B_R} H(x,D\bu)\,dx \leq c\left(\frac{R}{R_0}\right)^{n-\delta}\int_{B_{R_0}}H(x,D\bu)\,dx.
$$
Consequently, the mean energy satisfies
$$
\dashint_{B_R}H(x,D\bu)\,dx \leq cR^{-\delta}\int_{B_{R_0}}H(x,D\bu)\,dx,
$$
where $c$ now also depends on $R_0$. Combining this with the oscillation estimate, we obtain
$$
\dashint_{B_\rho}|D\bu-(D\bu)_{B_\rho}|^p\,dx \leq c\left[\left(\frac{\rho}{R}\right)^{\tilde{\alpha}p}+R^\sigma\left(\frac{R}{\rho}\right)^n\right]R^{-\delta}.
$$
We now choose $\delta < \frac{\sigma\tilde{\alpha}p}{n+\tilde{\alpha}p}$, so that there exists $\varepsilon > 0$ with
$$
\frac{\delta}{\tilde{\alpha}p}<\varepsilon<\frac{\sigma-\delta}{n},
$$
that is, $\varepsilon\tilde{\alpha}p-\delta>0$ and $\sigma-\delta-\varepsilon n>0$. Setting $\rho=R^{1+\varepsilon}$, we have $\frac{\rho}{R}=R^\varepsilon$ and $\frac{R}{\rho}=R^{-\varepsilon}$, and hence
$$
\dashint_{B_\rho}|D\bu-(D\bu)_{B_\rho}|^p\,dx \leq c\left[R^{\varepsilon\tilde{\alpha}p-\delta}+R^{\sigma-\delta-\varepsilon n}\right].
$$
We set $\gamma\coloneq\min\{\varepsilon\tilde{\alpha}p-\delta, \sigma-\delta-\varepsilon n\}>0$ and $\beta\coloneq\frac{\gamma}{p(1+\varepsilon)}\in(0,1)$. Since $R^\gamma=\rho^{\frac{\gamma}{1+\varepsilon}}$, we obtain
$$
\dashint_{B_\rho}|D\bu-(D\bu)_{B_\rho}|^p\,dx \leq c\rho^{p\beta},
$$
or equivalently,
$$
\int_{B_\rho}|D\bu-(D\bu)_{B_\rho}|^p\,dx \leq c\rho^{n+p\beta}.
$$
Every $\rho \in (0, R_0^{1+\varepsilon}]$ can be written as $\rho = R^{1+\varepsilon}$ with $R \in (0,R_0]$, and the constant $c$ is uniform for $x_0$ in compact subsets of $\Omega$. Therefore, Campanato's characterization of H\"older continuity yields
$$
D\bu\in C^{0,\beta}_{\mathrm{loc}}(\Omega;\mathbb R^{N\times n}).
$$
This completes the proof of Theorem \ref{thm : main theorem}.

\end{document}